\documentclass[10pt]{article}
\usepackage{amsmath,amssymb,wasysym, mathrsfs}
\usepackage[usenames,dvipsnames]{color}
\usepackage[usenames,dvipsnames]{color}
\usepackage{soul} 
\usepackage{cancel} 
\usepackage{graphicx}
\usepackage[left=3cm,right=3cm, top=3.3cm, bottom=3cm]{geometry}
\newtheorem{theo}{Theorem}[section]
\newtheorem{lemm}[theo]{Lemma}

\newtheorem{prop}[theo]{Proposition}
\newtheorem{defi}[theo]{Definition}
\newtheorem{remark}[theo]{Remark}
\def\proof {{\noindent \bf{Proof:\hspace{4pt}}}}
\def\endproof{\hfill$\square$\vspace{6pt}}
\numberwithin{equation}{section}
\title{
{\bf\Large  On the management fourth-order Schr\"{o}dinger-Hartree equation}}
\author{{Carlos Banquet} \\
{\small Universidad de C\'{o}rdoba, Departamento de Matem\'{a}ticas y Estad\'{\i}stica}\\
{\small A.A. 354, Monter\'{\i}a, Colombia.}\\
{\small \texttt{E-mail:cbanquet@correo.unicordoba.edu.co}}\vspace{.5cm}\\
{{\'Elder J. Villamizar-Roa}{\thanks{%
Corresponding author.}}}\\
{\small Universidad Industrial de Santander, Escuela de Matem\'{a}ticas}\\
{\small A.A. 678, Bucaramanga, Colombia.} \\
{\small \texttt{E-mail:jvillami@uis.edu.co}}}
\date{}
\begin{document}
\maketitle
\begin{abstract}
We consider the Cauchy problem associated to the fourth-order nonlinear Schr\"{o}dinger-Hartree equation with variable dispersion coefficients. The variable dispersion coefficients are 
assumed to be continuous or periodic and piecewise constant in time functions. We prove local and global well-posedness results for initial data in $H^s$-spaces. We also analyze the scaling limit of fast dispersion management and the convergence to a model with averaged dispersions.\newline

\end{abstract}

\section{Introduction}
A canonical model for propagation of intense laser beams in a bulk medium with Kerr nonlinearity is given by the nonlinear Schr\"{o}dinger equation
\begin{equation}\label{sch1}
i\psi_{t}(t,x,y)+\Delta \psi+\vert\psi\vert^2\psi=0.
\end{equation}
Two interesting situations related to model (\ref{sch1}) can be propounded. First, the role of introducing high order dispersion terms in the Schr\"{o}dinger equation (e.g. fourth oder), including models with variable dispersion coefficients, and second, the interactions described by the potential function $V(t,x_1,x_2)$ ($V=\vert \psi\vert^2$ in (\ref{sch1})). In the first case, as described in Fibich {\it et al.} \cite{Fibich}, the traditional derivation of the Schr\"{o}dinger equation in nonlinear optics comes from the nonlinear Helmholtz equation
\[
(\partial_{xx}+\partial_{yy}+\partial_{zz})E(x,y,z)+k^2E(x,y,z)=0,\ k^2=k_0^2\left(1+\frac{4n_2}{n_0}\vert E\vert^2\right),
\]
where $E$ is the electric field, $n_0$ is the index of refraction, $n_2$ is the Kerr coefficient and $k_0$ is wavenumber. Then, separating the slowly varying amplitude from the fast oscillations and changing to the following nondimensional variables
\begin{eqnarray*}
\tilde{x}=\frac{x}{r_0},\ \tilde{y}=\frac{y}{r_0},\ \tilde{t}=\frac{z}{2k_0r_0^2}, \ \psi(\tilde{x},\tilde{y},\tilde{t})=2r_0k_0\sqrt{\frac{n_2}{n_0}}E(x,y,z)e^{-ik_0z},
\end{eqnarray*}
one gets the nondimesional nonlinear Helmholtz equation
\begin{equation}\label{nondH}
\frac{\delta}{4}\psi_{\tilde{t}\tilde{t}}(\tilde{x},\tilde{y},\tilde{t})+i\psi_{\tilde{t}}+\Delta \psi+\vert\psi\vert^2\psi=0,
\end{equation}
where $\Delta=\partial_{\tilde{x}\tilde{x}}+\partial_{\tilde{y}\tilde{y}},$ $\delta=1/r_0^2k_0^2.$ Physically, $r_0$ is much larger than its wavelength $2\pi/k_0$ and therefore $\delta<<1,$ which permits to neglect the term $\psi_{\tilde{t}\tilde{t}}$ in (\ref{nondH}) and obtain the classical Schr\"{o}dinger equation. However, as point out in \cite{Fibich}, the neglected term $\psi_{\tilde{t}\tilde{t}}$ can becomes important, for instance, to prevent the collapse. Thus, we can consider the approximation of $\psi_{\tilde{t}\tilde{t}}\approx-\Delta^2\psi+O(\delta),$ where $\Delta^2$ is the biharmonic operator, in order to obtain
\begin{equation}\label{nondHc}
i\psi_{t}+\Delta \psi+\epsilon \Delta^2\psi+\vert\psi\vert^2\psi=0,\ \epsilon<0.
\end{equation}
Model (\ref{nondHc}) has been considered in a serie of papers, see for instance, \cite{Fibich,GuoCui4,GuoCui3,Guo,GuoCui5,Kar,Banquet} and references there in. Indeed, the nonlinear Schr\"{o}dinger equation with
mixed-dispersion
\begin{equation}\label{INLS}
i\partial _{t}\psi+\alpha \Delta \psi+\beta\Delta^2 \psi+|\psi|^\lambda \psi=0, \ \alpha, \beta=\mbox{constant},
\end{equation}
was initially considered by Karpman \cite{Kar} and Karpman
and Shagalov \cite{KarSha}, and it has been used as a model to
investigate the role played by the higher-order dispersion terms, in
formation and propagation of solitary waves in magnetic materials
where the effective quasi-particle mass becomes infinite. A particular case of (\ref{INLS}) corresponds to the Biharmonic equation 
\[
i\partial _{t}\psi+\beta \Delta^2 \psi+|\psi|^\lambda \psi=0,\ \beta=\mbox{constant,}
\]
which was introduced in \cite{Kar} and \cite{KarSha}, to take into account the role played by the higher
fourth-order dispersion terms in formation and propagation of intense laser beams in a bulk medium with Kerr nonlinearity, see also Ivano and Kosevich  \cite{Ivano}.\newline 

An additional point related to model (\ref{INLS}), also motivated by models in nonlinear optics, corresponds to the  case of dispersion managed $\alpha=\alpha(t),\beta=\beta(t),$ modelling varying dispersion along the fiber, which permits to balance the effects of nonlinearity and dispersion in such a way that stable nonlinear pulses (solitary waves) are supported over long distances (cf. \cite{Agrawal,Saut,Kurtzke, Lushnikov,Lushnikov2, Sulem, Zharnitsky}). See also Carvajal, Panthee and Scialom \cite{Carvajal}, and some references therein, to the case of a third-order nonlinear Schr\"{o}dinger equation with time-dependent coefficients.\newline

The second situation is related to the posible interactions described by the potential $V.$ An interesting interaction is given by the following coupled system
\begin{align}\label{FoSch0}
\left\{
\begin{array}{lc}
i\partial _{t}\psi(t,x)+\Delta \psi(t,x)=V(t,x)\psi(t,x), & x\in \mathbb{R}^{n},\ \ t\in \mathbb{R}, \\
\Delta V(t,x)=-\vert \psi(t,x)\vert^2, &  x\in \mathbb{R}^{n}, \ \ t\in\mathbb{R},
\end{array}
\right.
\end{align}
where $V(t,x)$ is a potential function. If $n\geq 3,$ the potential $V$ can be explicitly written as a solution of the Poisson equation (\ref{FoSch0})$_2$ as
\begin{equation}\label{pot}
V=C_n(\vert x\vert^{-(n-2)}*\vert \psi\vert^2),
\end{equation}
where $C_n$ is a constant which only depends on $n.$ Thus, substituting (\ref{pot}) into the Schr\"{o}dinger equation (\ref{FoSch0})$_1$ we obtain the so called Schr\"{o}dinger-Hartree equation
\begin{equation}\label{FoSch0b}
i\partial _{t}\psi(t,x)+\Delta \psi(t,x)=(|x|^{-(n-2)}*|\psi(t,x)|^{2})\psi(t,x),  x\in \mathbb{R}^{n},\ \ t\in \mathbb{R}.
\end{equation}
The nonlinearity in equation (\ref{FoSch0b}) has been generalized by considering the Hartree type nonlinearity $(|\cdot|^{-\lambda}*|\psi|^{2})\psi,$ $\lambda>0,$ which is relevant to describing several physical phenomena, as for instance, the dynamics of the mean-field limits of many-body quantum systems such as coherent states and condensates, the quantum transport in semiconductors superlattices, the study of mesoscopic structures in Chemistry, among others (cf. \cite{Elgart, Lieb,Yu}). From the mathematical point of view, some significative results on well-posedness in energy spaces has been obtained in \cite{Gao,Miao,Cho} and references therein.\newline

Based on the previous considerations, in this paper we study the Cauchy problem associated to  the following fourth-order Schr\"{o}dinger-Hartree equation 
with variable dispersion coefficients
\begin{equation}\label{FoSch}
\left\{
\begin{array}{lc}
i\partial _{t}u+\alpha(t)\Delta u+\beta(t) \Delta^2u+\theta(|x|^{-\lambda}*|u|^{2})u=0, & x\in \mathbb{R}^{n},\ \ t\in \mathbb{R}, \\
u(x,t_0)=u_{0}(x), &  x\in \mathbb{R}^{n},
\end{array}
\right.
\end{equation}
where the unknown $u(x,t)$ is a complex-valued function in space-time $
\mathbb{R}^n\times \mathbb{R}, n\geq 1,$ and  $u_0$ denotes the initial
data in $t_0\in\mathbb{R}.$

The coefficients $\alpha,\beta$ are real-valued functions which represent the variable dispersion coefficients. 
The constant $\theta\neq 0$ is a real coefficient which denotes the focusing or defocousing behavior (when diffraction and nonlinearity are working against or with each other). The nonlinearity coefficient $\lambda>0.$\newline

The general IVP (\ref{FoSch}) has not been considered in the literature. Thus, in this paper, we are interested in studying the well-posedness issues for the IVP (\ref{FoSch}) for given data based in the $L^2$-Sobolev spaces and, $\alpha,\beta$ continuous or piecewise constant periodic functions. The novelty of our results is summarizes in the following aspects: For initial data $u_0\in H^s(\mathbb{R}^n),$ $s\geq \mbox\{0,\lambda/2-2\}$ and $0<\lambda< n,$ we prove the existence of local in time solution $u\in C([-T+t_0,T+t_0];H^s(\mathbb{R}^n)).$ The proof is based on $L^p_tL^q_x$ properties of the linear propagator, as well as the Hardy-Littlewood-Sobolev inequality which allow us to control the Hartree nonlinearity. For initial data in $L^2,$ by using the conserved quantity $\Vert u(t)\Vert_{L^2(\mathbb{R}^n)}=\Vert u_0\Vert_{L^2(\mathbb{R}^n)}$ we are able to extend the local solution globally. We also prove the existence of global solution in $H^1$ by combining the  $L^2$-conservative law, the local well-posedness in $H^1$ and argument of blow up alternative. If the nonlinearity is given by $\theta |u|^{2}u,$ we also analyze the existence of global solution in $H^s,$ $s\geq 0.$
Finally, we will address the scaling limit to fast dispersion management, that is, for each $\epsilon>0,$ we consider the $\epsilon$-scaled fourth-order nonlinear Schr\"{o}dinger equation by making $\beta_\epsilon(t)=\beta(\frac{t}{\epsilon}),$ $\alpha_\epsilon(t)=\alpha(\frac{t}{\epsilon}),$ and then, we analyze the scaling limit $\epsilon\rightarrow 0^+$ of the solutions.  \newline

This article is organized as follows. In Section 2, we establish some linear estimates which are fundamental for obtaining our results of local and global mild solutions. In Section 3, we prove the existence of local solutions in $H^s$ for $s\geq \lambda/2$. In Section 4, we analyze the existence of local solutions in $H^s$ for $ \mbox\{0,\lambda/2-2\}\leq s<\lambda/2.$ In Section 5, we prove some results of global existence. Finally, in Section 6, we give a result about the scaling limit to fast dispersion management.

\section{Linear propagator}
Before studying the nonlinear Cauchy problem we
give some properties of the linear problem associated to (\ref{FoSch}), which is given by
\begin{equation}\label{LinFoScha}
\left\{
\begin{array}{lc}
i\partial _{t}u+\alpha(t)\Delta u+\beta(t) \Delta^2u=0, & x\in \mathbb{R}^{n},\ \ t\in \mathbb{R}, \\
u(x,t_0)=u_{0}(x), &   x\in \mathbb{R}^{n}.
\end{array}
\right.
\end{equation}
For $\alpha$ and $\beta$ being integrable functions, we define the {\it cumulative dispersions}
$A(t_0,t)$ and $B(t_0,t)$ on the closed interval $[t_0,t]$ by
$$A(t,r)=\int_{r}^{t}\alpha(\tau)d\tau \ \ \text{and}\ \ B(t,r)=\int_{r}^{t}\beta(\tau)d\tau.$$
We denote by $U_{\alpha, \beta}(t,t_0)$ the linear propagator which describes the solution $u(x,t)$ of (\ref{LinFoScha}). It holds that
\[
 U_{\alpha,\beta}(t,t_0) u_0(x)=
\left[e^{-i\vert\xi\vert^2A(t,t_0)+i\vert\xi\vert^4B(t,t_0)}\widehat{u_0}(\xi)\right]^{\vee}(x).
\]
Then, for any $t,r,l \in \mathbb{R},$ it holds
\begin{equation}\label{grup1a}
U_{\alpha,\beta}(t,r)=U_{\alpha,\beta}(t,l)U_{\alpha,\beta}(l,r)
\end{equation}
and
\begin{equation}\label{grup1ba}
U_{\alpha,\beta}(t,r)=U_{\alpha,\beta}(r,t)^{-1}=U_{-\alpha,-\beta}(r,t).
\end{equation}
We will use the notation $U(t,t_0):=U_{\alpha,\beta}(t,t_0)$ and $U(t):=U(t,0).$ Then, (\ref{grup1a})-(\ref{grup1ba}) imply that
\[U(t,r)=U(t)U(r)^{-1}.\]
For each $s\in \mathbb{R},$ the propagator $U(t,t_0)$ is an isometry on $H^s(\mathbb{R}),$ that is, for any $f\in H^s(\mathbb{R})$ it holds
\begin{equation}\label{grup2a}
\|U(t,t_0)f\|_{H^s}=\|f\|_{H^s}=\Vert \langle\xi\rangle^s\hat{f}(\xi)\Vert_{L^2(\mathbb{R}^n)}.
\end{equation}
However, $U(t,r)\neq U(t-r,0),$ since
\begin{eqnarray*}
\int_r^t\alpha(\tau)d\tau=\int_0^{t-r}\alpha(\tau+r)d\tau\neq\int_0^{t-r}\alpha(\tau)d\tau,\ \int_r^t\beta(\tau)d\tau=\int_0^{t-r}\beta(\tau+r)d\tau\neq\int_0^{t-r}\beta(\tau)d\tau,
\end{eqnarray*}
unless $\alpha,\beta$ be constant functions. Thus, $U(t,r)$ is not a group.
\begin{lemm}\label{VanCorpus1a} Let $\phi(\xi)=b\vert\xi\vert^4+a\vert\xi\vert^2,$ $\xi\in \mathbb{R}^n,$ $a,b\in\mathbb{R}$ and $b\neq 0.$ For each $f\in \mathscr{S}(\mathbb{R}^n),$ consider the operator
\[\mathcal{L}f(x)=\int_{\mathbb{R}^n}e^{i\phi(\xi)+ix\cdot \xi}\widehat{f}(\xi)d\xi=(f*I_{\phi})(x),\]
where
\[I_{\phi}(x)=\int_{\mathbb{R}^n}e^{i\phi(\xi)+ix\cdot \xi} d\xi.\]
Then, for $0\leq \theta \leq 1,$ $\frac 1p+\frac 1{p'}=1$ and $\frac 1p=\frac{1-\theta}{2}$ it holds
\[\|\mathcal{L}f\|_{L^p(\mathbb{R}^n)}\leq C |b|^{-\frac{n\theta}{4}}\|f\|_{L^{p'}(\mathbb{R}^n)}.\]
\end{lemm}
\proof From Plancherel's Theorem we get
\begin{equation}\label{L2L2a}
\|\mathcal{L}f\|_{L^2(\mathbb{R}^n)}=\|f\|_{L^2(\mathbb{R}^n)}.
\end{equation}
Assuming for a moment that
\begin{equation}\label{EstIphia}
|I_{\phi}(x)|\leq C|b|^{-\frac{n}{4}},
\end{equation}
we can conclude that
\begin{equation}\label{LinffL1a}
\|\mathcal{L}f\|_{L^{\infty}(\mathbb{R}^n)}\leq C|b|^{-\frac{n}{4}}\|f\|_{L^1(\mathbb{R}^n)}.
\end{equation}
Thus, the result follows directly from (\ref{L2L2a}), (\ref{LinffL1a}) and the Riesz-Thorin interpolation Theorem. In order to finish the proof, we just have to show (\ref{EstIphia}). We begin taking $n=1.$ Define $h(\xi)=a\vert\xi\vert^2+b\vert\xi\vert^4+x\xi.$ Since $|h^{(4)}(\xi)|=24\vert b\vert,$ by using Van der Curput's lemma, we arrived at
\[|I_{\phi}(x)|\leq C|b|^{-\frac{1}{4}}.\]
The result for $n\geq 2$ can be obtained from Theorem 2 in Cui \cite{Cui}, taking $\mu=0,$ $m=4,$ $P(\xi)=\phi(\xi),$ $q=p$ and $p=p'.$  Indeed, it is clear that $\left(\frac{1}{p'},\frac{1}{p}\right)\in\Delta_0,$ where we use the notation $\Delta_0$ to denote the region in the $(\frac{1}{p}, \frac{1}{q})$ plane occupied by the quadrilateral $R_0P_0BQ_0,$ comprising the apices $R_0=(\frac12,\frac12),$  $B=(1,0)$ and all edges $P_0B, BQ_0, P_0R_0,$ and $Q_0R_0$, but not comprising the apices $P_0=(\frac23,0)$ and $Q_0=(1,\frac13).$
Also, $\phi$ is a real elliptic polynomial, with $\phi(0)=0,$ $\mathrm{deg}(\phi)=4;$ moreover, denoting by $P_4(\xi)=b|\xi|^4,$ we obtain that the Hessian 
\[HP_4(\xi)=3b^n4^n\left(\sum_{i=1}^{n}\xi_i^2\right)^n, \]
for $\xi\in \mathbb{R}^n\setminus\{0\},$ is nondegenerate, which implies the desired result for $n\geq 2.$
\endproof

Next lemmas will be useful in order to estimate the nonlinearity in (\ref{FoSch}).
\begin{lemm}[Hardy inequality]\cite{KatoLib} \label{MaxSti}Let $0<\lambda <n.$ Then there exists $c=c(n,\lambda)>0$ such that for all $f\in{\dot{H}^{\frac{\lambda}{2}}},$
\[\||x|^{-\lambda}*|f|^{2}\|_{L^{\infty}}\leq C(n,\lambda) \|f\|^2_{\dot{H}^{\frac{\lambda}{2}}}.\]
\end{lemm}
\begin{lemm}[Hardy-Littlewood-Sobolev] \cite{LiebLoss}\label{HLSineq}
Let $0< \lambda<n ,$ $1< p<q<\infty$ with $\frac{1}{q}=\frac{1}{p}+\frac{\lambda}{n}-1.$ Then it holds
\[
\| |x|^{-\lambda}*f\|_{L^q(\mathbb{R}^n)}\leq C\|f\|_{L^p(\mathbb{R}^n)}.
\]
Moreover, if $l,r>1$ such that $\frac{1}{r}+\frac{1}{l}+\frac{\lambda}{n}=2$ and $f\in L^r, g\in L^l,$ then
\[
\left\vert\int_{\mathbb{R}^n}\int_{\mathbb{R}^n} f(x)|x-y|^{-\lambda}g(y)dxdy\right\vert\leq C\|f\|_{L^r}\|g\|_{L^l}.
\]
\end{lemm}

\begin{lemm}\cite{Kato} \label{LeibRule}For any $s\geq 0,$ we have
\[ \Vert D^s(uv)\Vert_{L^r}\apprle  \Vert D^su\Vert_{L^{p_1}} \Vert v\Vert_{L^{q_2}}+ \Vert v\Vert_{L^{q_1}} \Vert D^sv\Vert_{L^{p_2}},\]
where $D^s=(-\Delta)^{s/2}$ and $\frac{1}{r}=\frac{1}{p_1}+\frac{1}{q_2}=\frac{1}{q_1}+\frac{1}{p_2},$ $p_i\in(1,\infty),$ $q_i\in(1,\infty],$ for $i=1,2.$
\end{lemm}


\subsection{Linear propagator with piecewise constant dispersion}
In this subsection we establish some Strichartz estimates related to the linear propagator with $\alpha$ and $\beta$ piecewise constant functions of kind
\begin{equation}\label{alpha_beta}
\alpha(t)=\left\{
\begin{array}{lc}
\alpha^+,\ \ 0<t\leq t_+, \\
-\alpha^{-}, \ \ t_+-T_1<t\leq 0,
\end{array}
\right.
\ \ \beta(t)=\left\{
\begin{array}{lc}
\beta^+,\ \ 0<t\leq \tau_+, \\
-\beta^{-}, \ \ \tau_+-T_2<t\leq 0,
\end{array}
\right.
\end{equation}
where $\alpha^+,\alpha^-,\beta^+$ and $\beta^-$ are positive constants, $t_+\in (0,T_1),$ $\tau_+\in(0,T_2),$ $\alpha(t+T_1)=\alpha(t),$ and $\beta(t+T_2)=\beta(t)$ for all $t\in \mathbb{R}$ (see Figure 1).\newline
\\
\vspace{0.1cm}
\begin{minipage}{\linewidth}
\begin{center}
{\includegraphics[scale=1.0]{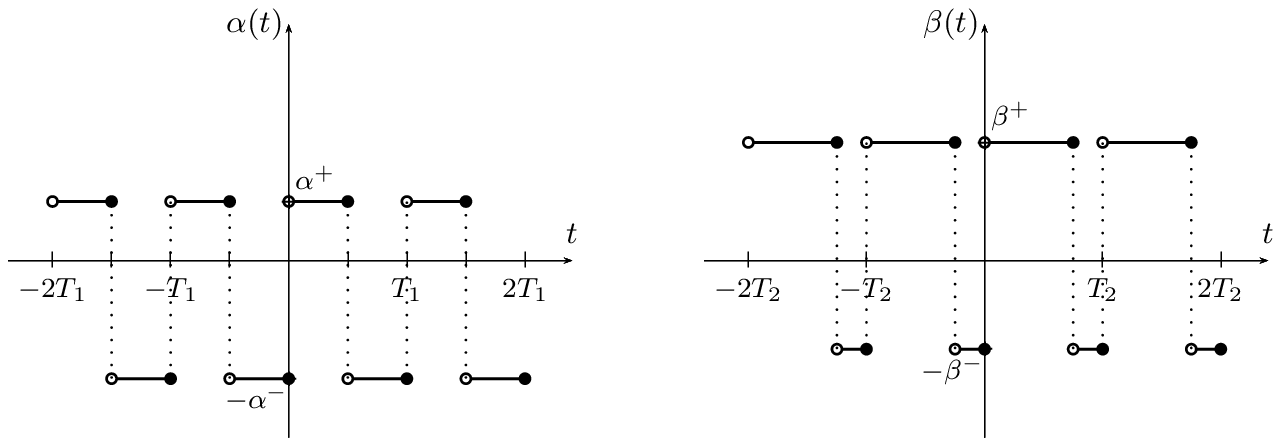}}\\
\end{center}
\end{minipage}
\begin{center}
\vspace{0.1cm}
\small{{\bf Fig. 1}{ Sketch of the dispersion functions} }
 \end{center}
Without loss of generality we assume that $T_2=1.$ Then, we can split the real line as follows
\begin{equation}\label{descomp}
\mathbb{R}=\bigcup_{m\in\mathbb{Z}}(m,m+\tau_+]\cup (m+\tau_+,m+1].
\end{equation}
We have the following estimate:
\begin{lemm}\label{VanCorput2}
Let $\tau_+\in(0,1)$ and $t,r\in(m,m+\tau_+]$ or $t,r\in(m+\tau_+,m+1],$ for $m$ an integer. Then, for $r\neq t,$ it holds
\[
\|U(t,r)f\|_{L^p(\mathbb{R}^n)}\leq C|t-r|^{-\frac{n\theta }{4}} \| f\|_{L^{p'}(\mathbb{R}^n)},
\]
where $0\leq \theta \leq 1,$ $\frac 1p+\frac 1{p'}=1$ and $\frac 1p=\frac{1-\theta}{2}.$
\end{lemm}
\proof The proof is similar to the proof of Lemma \ref{VanCorpus1a} by taking $a=A(t,r)=\int_{r}^{t}\alpha(\tau)d\tau$ and $b=B(t,r)=\int_{r}^{t}\beta(\tau)d\tau.$
Note that if $t, r\in (m,m+\tau_+]$ or $t,r\in(m+\tau_+,m+1]$ we have $b=\pm(t-r)\beta^{\pm};$ consequently,
\[|I_{\phi}(x)|\leq C|t-r|^{-\frac{n}{4}}.\]
\endproof
\begin{defi}
A pair $(p,q)$ is said {\it admisible} if 
\begin{eqnarray*}
\left\{\begin{array}{lr}
        2\leq p<\frac{2n}{n-4} & \text{if}\  n\geq5\\
        2\leq p<\infty & \text{if}\  n=4\\
        2\leq p \leq\infty & \text{if}\  n\leq 3
        \end{array}\right\} \ \ \text{and}\ \ \frac {4}{q}={n}\left(\frac12-\frac {1}{p}\right).
\end{eqnarray*} 
\end{defi}
The linear propagator $U(t,r)$ satisfies some Strichartz estimates on each time-interval $(m,m+\tau_+]$ and $(m+\tau_+,m+1]$. More exactly, we have the following result.
\begin{prop}\label{pr1a} Let $m\in\mathbb{Z},$ $\tau_+\in(0,1)$ be given. Then, for each admisible pairs $(q,p),$ $(q_1,p_1),$ $(q_2,p_2),$ and each time-interval $I_m\subset (m,m+\tau_+]$ or $I_m\subset (m+\tau_+,m+1],$ it holds:
\begin{enumerate}
 \item There exists $C_1=C_1(p,I_m)>0$ such that for $t_0\in I_m$ and for any $f\in L^2(\mathbb{R}^n)$ it holds
\begin{equation}\label{IneLin2a}
\|U(t,t_0)f\|_{L^q(I_m;L^p(\mathbb{R}^n))}\leq C\|f\|_{L^2(\mathbb{R}^n)}.
\end{equation}
\item There exists $C_2=C_2(p_1,p_2,I_m)>0$ such that for $t_0\in I_m$ and any $g\in {L^{q'_2}(I_m;L^{p'_2})}$ it holds
\begin{equation}\label{StriSti1}
\left\Vert \int_{I_m\cap \{\tau\leq t\}} U(t,\tau)g(\tau)d\tau\right\Vert_{L^{q_1}(I_m;L^{p_1}(\mathbb{R}^n))}\leq C_2 \Vert g\Vert_{L^{q'_2}(I_m;L^{p'_2})}.
\end{equation}
\end{enumerate}
\end{prop} 

\proof  In order to prove (\ref{IneLin2a}) we use a duality argument. For that, it is enough to show that for any $g\in {L^{q'}(I_m;L^{p'}({\mathbb{R}^n})})
$ it holds
\[
\left| \int_{I_m} \int_{\mathbb{R}^n} g(x,t)\overline{U(t,t_0)f(x)}dxdt\right|\leq C\|f\|_{L^2(\mathbb{R}^n)}\|g\|_{L^{q'}(I_m;L^{p'}({\mathbb{R}^n}))}.
\]
From Fubini's Theorem, Cauchy-Schwarz inequality and Plancherel's identity, we arrived at
\[
\left| \int_{I_m} \int_{\mathbb{R}^n}g(x,t)\overline{U(t,t_0)f(x)}dxdt\right|\leq C\|f\|_{L^2(\mathbb{R}^n)}\left\|\int_{I_m}{U}^{-1}(t,t_0)g(\cdot, t)dt\right\|_{L^2(\mathbb{R}^n)}.
\]
Following the arguments in Tomas \cite{tomas}, p.477, we get
\begin{align}\label{IneTho1}
\left\|\int_{I_m}{U}^{-1}(t,t_0)g(\cdot, t)dt\right\|^2_{L^2(\mathbb{R}^n)}&=\int_{\mathbb{R}^n}\int_{I_m}
\overline{g(x,t)}\int_{I_m}U(t,\tau)g(x,\tau)d\tau dt dx\notag\\
&\leq \|g\|_{L^{q'}(I_m;L^{p'}(\mathbb{R}^n))}\left\|\int_{I_m}U(\tau,t)g(\cdot, \tau)d\tau\right\|_{L^q(I_m;L^{p}(\mathbb{R}^n))}.
\end{align}
In order to conclude the proof it is enough to show that
\begin{equation}\label{IneTho2}
\left\|\int_{I_m}U(\tau,t)g(\cdot,\tau)d\tau\right\|_{L^{q}(I_m;L^{p}(\mathbb{R}^n))}\leq C \|g\|_{L^{q'}(I_m;L^{p'}(\mathbb{R}^n))}.
\end{equation}
Note that for $t,r,\tau \in I_m,$ we have $|B(t,r)-B(\tau,r)|= \beta^{\pm}|t-\tau|,$ then by Fubini's Theorem, Lemma \ref{VanCorput2} and the Hardy-Littlewood-Sobolev inequality, we have
\begin{eqnarray}
\left\|\int_{I_m}U(\tau,t)g(\cdot,\tau)d\tau\right\|_{L^{q}(I_m;L^{p}(\mathbb{R}^n))}&\leq & \left\|\int_{I_m}\|U(\tau,t)g(\cdot,\tau)\|_{L^{p}(\mathbb{R}^n)}d\tau\right\|_{L^{q}(I_m)}\nonumber\\
&\leq & \left\|\int_{I_m}\frac{\|g(\cdot,\tau)\|_{L^{p'}(\mathbb{R}^n)}}{|t-\tau|^{\frac{n\theta}{4}}}d\tau\right\|_{L^{q}(I_m)}
\leq C \|g\|_{L^{q'}(I_m;L^{p'}(\mathbb{R}^n))}.\nonumber
\end{eqnarray}
Now, from (\ref{IneTho1}) and (\ref{IneTho2}) we obtain
\[
\left\|\int_{I_m}U^{-1}(t,t_0)g(\cdot,t)dt\right\|_{L^{2}(\mathbb{R}^n)}\leq C \|g\|_{L^{q'}(I_m;L^{p'}(\mathbb{R}^n))}.
\]
This finishes the proof of (\ref{IneLin2a}).

Now we prove (\ref{StriSti1}). By hypothesis, the points $(p_2,q_2)$ and $(p_1,q_1)$ are in the segment of the line connecting $P=\left(\frac{1}{2},0\right)$ with $Q=\left( \frac{1}{p(n)}, \frac{n}{8}-\frac{1}{4p(n)}\right).$ Then, $p(n)=\infty$ if $n=1,2,3,4$ and, $P(n)=\frac{2n}{n-4}$ if $n\geq 5.$ Therefore, without loss of generality we can assume that $p_2\in [2,p_1).$ This implies that $q_2\in [q_1,\infty).$ Combining inequalities (\ref{IneTho1})-(\ref{IneTho2}) we arrive at
\[
\left\|\int_{t_0}^tU(t,\tau)g(\cdot,\tau)d\tau\right\|_{L^{q_1}(I_m;L^{p_1}(\mathbb{R}^n))}\leq C \|g\|_{L^{q_1'}(I_m;L^{p_1'}(\mathbb{R}^n))}.
\]
and 
\begin{align*}
\sup_{t\in I_m}\left\|\int_{t_0}^tU(t,\tau)g(\cdot,\tau)d\tau\right\|_{L^{2}(\mathbb{R}^n)}&=\sup_{t\in I_m}\left\|U(t,t_0)\int_{t_0}^tU(t_0,\tau)g(\cdot,\tau)d\tau\right\|_{L^{2}(\mathbb{R}^n)}\notag \\
&= \sup_{t\in I_m}\left\|\int_{t_0}^tU(t_0,\tau)g(\cdot,\tau)d\tau\right\|_{L^{2}(\mathbb{R}^n)}\leq C \|g\|_{L^{q_1'}(I_m;L^{p_1'}(\mathbb{R}^n))}.
\end{align*}
From the last estimates and an interpolation argument we have
\[
\left\|\int_{t_0}^tU(t,\tau)g(\cdot,\tau)d\tau\right\|_{L^{q_2}(I_m;L^{p_2}(\mathbb{R}^n))}\leq C \|g\|_{L^{q_1'}(I_m;L^{p_1'}(\mathbb{R}^n))}.
\]
From the last inequality and an argument of duality, we obtain
\[
\left\|\int_{t_0}^tU(t,\tau)g(\cdot,\tau)d\tau\right\|_{L^{q_1}(I_m;L^{p_1}(\mathbb{R}^n))}\leq C \|g\|_{L^{q_2'}(I_m;L^{p_2'}(\mathbb{R}^n))},
\]
which yields the result. 
\endproof

\subsection{Linear propagator with continuous dispersion}
In this section we analyze the propagator associated to the linear problem (\ref{LinFoScha}), where the dispersion functions $\alpha, \beta\in C([-T+ t_0 , t_0 + T])$ such that $\beta(t)\neq 0$ for all $t\in \mathbb{R}.$ Below, we establish some Strichartz estimates related to the linear propagator $U(t,r)$. From now on, we will use the notation $$\Vert f\Vert_{L_T^{q}L_x^{p}}:=\Vert f\Vert_{L^q([-T+t_0, t_0 + T];L^p(\mathbb{R}^n))},\ T>0.$$
 
 \begin{prop}\label{pr1} For each admisible pairs $(q,p),$ $(q_1,p_1),$ $(q_2,p_2),$ it holds:
\begin{enumerate}
 \item There exists $C_1=C_1(p)>0$ such that for any $t\in [-T+t_0, t_0 + T]$ and $f\in L^2(\mathbb{R}^n)$ it holds
\begin{equation}\label{IneLinCont}
\|U(t,r)f\|_{L_t^qL_x^p}\leq C\|f\|_{L^2(\mathbb{R}^n)}.
\end{equation}
\item There exists $C_2=C_2(p_1,p_2)>0$ such that for any $F\in {L^{q'_2}(I;L^{p'_2})},$ $I=[-T+t_0, t_0 + T],$ it holds
\begin{equation}\label{StriSti2}
\left\Vert \int_{I} U(t,t_0)F(\tau)d\tau\right\Vert_{L^{q_1}_TL_x^{p_1}}\leq C_2 \Vert F\Vert_{L_T^{q'_2}L_x^{p'_2}}.
\end{equation}
\end{enumerate}
\end{prop} 
\proof In order to proof (\ref{IneLinCont}) we again use a duality argument. Thus, it is enough to show that for any $g\in L_t^{q'}L_x^{p'}$ it holds
\[
\left| \int_{\mathbb{R}} \int_{\mathbb{R}^n} g(x,t)\overline{U(t,r)f(x)}dxdt\right|\leq C\|f\|_{L^2}\|g\|_{L_t^{q'}L_x^{p'}}.
\]
From Fubini's Theorem, Cauchy-Schwarz inequality and Plancherel's identity, we arrived at
\[
\left| \int_{\mathbb{R}} \int_{\mathbb{R}^n}g(x,t)\overline{U(t,r)f(x)}dxdt\right|\leq C\|f\|_{L^2}\left\|\int_{\mathbb{R}}{U}^{-1}(t,r)g(\cdot, t)dt\right\|_{L_x^2(\mathbb{R}^n)},
\]
Similarly to the proof of Proposition \ref{pr1a} we get
\begin{align*}
\left\|\int_{\mathbb{R}}{U}^{-1}(t,r)g(\cdot, t)dt\right\|^2_{L_x^2(\mathbb{R}^n)}\leq \|g\|_{L_t^{q'}L_x^{p'}}\left\|\int_{\mathbb{R}}U(\tau,t)g(\cdot, \tau)d\tau\right\|_{L_t^{q}L_x^{p}}.
\end{align*}
In order to conclude the proof it is enough to show that
\[
\left\|\int_{\mathbb{R}}U(\tau,t)g(\cdot,\tau)d\tau\right\|_{L_t^{q}L_x^{p}}\leq C \|g\|_{L_t^{q'}L_x^{p'}}.
\]
Since $|\beta(t)|\geq \delta_1>0$ for all $t\in \mathbb{R},$ then
\[
|B(t,r)|\geq \delta_1|t-\tau|.
\]
Therefore, using the Fubini's Theorem, following Lemma \ref{VanCorpus1a} and the Hardy-Littlewood-Sobolev inequality (Lemma \ref{HLSineq}), we obtain
\begin{eqnarray*}
\left\|\int_{\mathbb{R}}U(\tau,t)g(\cdot,\tau)d\tau\right\|_{L_t^{q}L_x^{p}}\leq \left\|\int_{\mathbb{R}}\|U(\tau,t)g(\cdot,\tau)\|_{L_x^{p}}d\tau\right\|_{L_t^{q}}\leq \left\|\int_{\mathbb{R}}\frac{\|g(\cdot,\tau)\|_{L_x^{p'}}}{|t-\tau|^{\frac{n\theta}{4}}}d\tau\right\|_{L_t^{q}}
\leq C \|g\|_{L_t^{q'}L_x^{p'}}.
\end{eqnarray*}
This rest of the proof of (\ref{IneLinCont}) is similar to Proposition \ref{pr1a}. On the other hand, a duality argument analogous to the proof of (\ref{StriSti1}) permits to prove (\ref{StriSti2}).
\endproof

\section{Local well-posedness in $H^s(\mathbb{R}^n)$  with $s\geq \frac{\lambda}{2}$}
In this section we prove the local existence in $H^s(\mathbb{R}^n)$  with $s\geq \frac{\lambda}{2}.$ We assume that the variable dispersion $\alpha,\beta$ verifies either 
$\alpha,\beta\in C([-T+t_0,T+t_0])$  with $\beta(t)\neq 0$, for all $t\in[-T+t_0,T+t_0],$ or $\alpha,\beta$ are periodic piecewise constants. Results of local well-posedness in the case of   the Schr\"{o}dinger-Hartree equation with constant dispersion ($\alpha(t)=$constant and $\beta(t)=0$) were obtained in Miao {\it et al} \cite{Miao}.
The proof is obtained through the contraction mapping argument. For that, as usual, we consider the solution of (\ref{FoSch}) via the Duhamel's formula which is given by
\begin{equation}\label{duhamel0}
u(t)=U(t,t_0)u_0+i\theta\int_{t_0}^tU(t,\tau)\{(|x|^{-\lambda}*|u(\tau)|^{2})u(\tau)\}d\tau.
\end{equation}

\subsection{Local well-posedness with continuous dispersion}
\begin{theo}\label{teo1}
Let $n\geq 1,$ $0<\lambda<n,$  $u_0\in H^s(\mathbb{R}^n),$ $s\geq \frac{\lambda}{2},$ and $\alpha,\beta\in C([-T+t_0,T+t_0])$  with $\beta(t)\neq 0$, for all $t\in[-T+t_0,T+t_0].$ Then there exists $T_0=T_0(\Vert u_0 \Vert_{H^s})\leq T$ and a unique solution $u$ of (\ref{duhamel0}) in the class $C([-T_0+t_0,T_0+t_0];H^s(\mathbb{R}^n))$ verifying $\Vert u\Vert_{L_{T_0}^{\infty}H^s}\leq C\Vert u_0\Vert_{H^s}.$
\end{theo}
\proof Consider the mapping
\begin{equation}\label{fixedpoin}
\Phi_1(u)(t)=U(t,t_0)u_0+i\theta\int_{t_0}^tU(t,\tau)\{(|x|^{-\lambda}*|u(\tau)|^{2})u(\tau)\}d\tau.
\end{equation}
Let $R>0$ and $(X_{T,R}^s,d)$ be the complete metric space 
\[ X_{T,R}^s=\left\{ u\in L_T^{\infty}(H^s(\mathbb{R}^n)): \Vert u\Vert_{L_T^{\infty}H^s}\leq R \right\},
\]
with metric $d(u,v)=\Vert u-v\Vert_{L_T^{\infty}L^2}.$\\ 

From (\ref{grup2a}), Lemmas \ref{MaxSti}, \ref{HLSineq} and \ref{LeibRule}, and the Sobolev embedding, we obtain
\begin{align}
\|\Phi_1(u)\|_{H^s}&\leq  \|U(t,t_0)u_0\|_{H^s}+\left\|\theta\int_{t_0}^tU(t,\tau)\{(|x|^{-\lambda}*|u(\tau)|^{2})u(\tau)\}d\tau\right\|_{H^s}\nonumber\\
&\apprle \|u_0\|_{H^s}+\int_{t_0-T}^{t_0+T}\left\|(|x|^{-\lambda}*|u(\tau)|^{2})u(\tau)\right\|_{H^s}d\tau\nonumber\\
&\apprle  \|u_0\|_{H^s}+T\left\|(|x|^{-\lambda}*|u|^{2})u\right\|_{L^{\infty}_TH^s}\nonumber\\
&\apprle  \|u_0\|_{H^s}+T\left\||x|^{-\lambda}*|u|^{2}\|_{L^{\infty}_TL^{\infty}_x} \|u\right\|_{L^{\infty}_TH^s}+T\||x|^{-\lambda}*|u|^{2}\|_{L^{\infty}_TH^s_{\frac{2n}{\lambda}}}
\|u\|_{L^{\infty}_TL^{\frac{2n}{n-\lambda}}_x}\nonumber\\
&\apprle  \|u_0\|_{H^s}+T\|u\|^2_{L^{\infty}_T\dot{H}^{\lambda/2}} \|u\|_{L^{\infty}_TH^s}+T\||u|^{2}\|_{L^{\infty}_TH^s_{\frac{2n}{2n-\lambda}}}
\|u\|_{L^{\infty}_TL^{\frac{2n}{n-\lambda}}_x}\nonumber\\
&\apprle  \|u_0\|_{H^s}+T\|u\|^2_{L^{\infty}_T\dot{H}^{\lambda/2}} \|u\|_{L^{\infty}_TH^s}+T\|u\|_{L^{\infty}_TH^s}
\|u\|^2_{L^{\infty}_TL^{\frac{2n}{n-\lambda}}_x}\nonumber\\
&\apprle  \|u_0\|_{H^s}+T\|u\|^2_{L^{\infty}_T\dot{H}^{\lambda/2}} \|u\|_{L^{\infty}_TH^s}\nonumber\\
&\apprle  \|u_0\|_{H^s}+T \|u\|^3_{L^{\infty}_TH^s}.\label{es1}
\end{align}
Therefore, $\|\Phi_1(u)\|_{L^{\infty}_TH^s}\apprle \|u_0\|_{H^s}+T \|u\|^3_{L^{\infty}_TH^s}.$ If we choose $R$ and $T_0\leq T$ such that $C\Vert u_0\Vert_{H^s}\leq \frac R2$ and $CT_0R^2\leq \frac 12,$ we have that $\Phi_1$ maps $X_{T_0,R}^s$ to itself. Now, from the H\"{o}lder inequality, Lemma \ref{MaxSti} and the Sobolev embedding, we get
\begin{align*}
\|\Phi_1(u)-\Phi_1(v)\|_{L^{\infty}_TL^2_x}&\apprle T\left\|(|x|^{-\lambda}*|u|^{2})u-(|x|^{-\lambda}*|v|^{2})v\right\|_{L^{\infty}_TL^2}\\
&\apprle T\left\|(|x|^{-\lambda}*|u|^{2})(u-v)\right\|_{L^{\infty}_TL^2}+T\left\|(|x|^{-\lambda}*(|u|^{2}-|v|^{2}))v\right\|_{L^{\infty}_TL^2}\\
&\apprle T\left\|(|x|^{-\lambda}*|u|^{2})\right\|_{L^{\infty}_TL^{\infty}}  \|u-v\|_{L^{\infty}_TL^2}\\
&+T\left\||x|^{-\lambda}*(|u|^{2}-|v|^{2})\right\| _{L^{\infty}_TL^{\frac{2n}{\lambda}}}\|v\|_{L^{\infty}_TL^{\frac{2n}{n-\lambda}}}\\
&\apprle T\|u\|^2_{L^{\infty}_T\dot{H}^{\lambda/2}} \|u-v\|_{L^{\infty}_TL^2}+T\||u|^{2}-|v|^{2}\| _{L^{\infty}_TL^{\frac{2n}{2n-\lambda}}}\|v\|_{L^{\infty}_T\dot{H}^{\lambda/2}}\\
&\apprle T\|u\|^2_{L^{\infty}_T\dot{H}^{\lambda/2}} \|u-v\|_{L^{\infty}_TL^2}+T\|u-v\| _{L^{\infty}_TL^2}\|u+v\| _{L^{\infty}_TL^{\frac{2n}{n-\lambda}}}\|v\|_{L^{\infty}_T\dot{H}^{\lambda/2}}\\
&\apprle T\|u-v\|_{L^{\infty}_TL^2}\left( \|u\|^2_{L^{\infty}_T\dot{H}^{\lambda/2}}  + \|u+v\| _{L^{\infty}_TL^{\frac{2n}{2n-\lambda}}}\|v\|_{L^{\infty}_T\dot{H}^{\lambda/2}}\right)\\
&\apprle TR^2\|u-v\|_{L^{\infty}_TL^2}.
\end{align*}
Thus, if we take $T_0\leq T$ small enough, $\Phi_1$ is a contraction. Consequently, $\Phi_1$ has a unique fixed point at $X_{T_0,R}^s$ which is solution of (\ref{duhamel0}).
Finally, we will prove the time-continuity of the solution. For that, let $t_1\in[-T_0+t_0,t_0+T_0].$ We will show that
\begin{equation}\label{lim2}
\lim_{t\rightarrow t_1}\Vert u(t)-u(t_1)\Vert_{H^s}=0.
\end{equation}
From integral equation (\ref{duhamel0}) we have
\begin{equation}\label{duhamel2}
u(t_1)=U(t_1,t_0)u_0+i\theta\int_{t_0}^{t_1}U(t_1,\tau)\{(|x|^{-\lambda}*|u(\tau)|^2)u(\tau)\}d\tau.
\end{equation}
Then, taking the $H^s$-norm of the difference between (\ref{duhamel0}) and (\ref{duhamel2}) we get
\begin{align}
&\Vert u(t)-u(t_1)\Vert_{H^s}\leq \Vert U(t,t_0)u_0-U(t_1,t_0)u_0\Vert_{H^s}\nonumber\\
&\hspace{0.5cm} +\left\Vert \theta\int_{t_0}^{t}U(t,\tau)\{(|x|^{-\lambda}*|u(\tau)|^2)u(\tau)\}d\tau-\theta\int_{t_0}^{t_1}U(t,\tau)\{(|x|^{-\lambda}*|u(\tau)|^2)u(\tau)\}d\tau\right\Vert_{H^s}\nonumber\\
&\hspace{0.5cm}:=J_1+J_2.\nonumber
\end{align}
Notice that
\[
J_1=\left\Vert\langle\xi\rangle^s(e^{-i\xi^2A(t,t_0)+i\xi^4B(t,t_0)}-e^{-i\xi^2A(t_1,t_0)+i\xi^4B(t_1,t_0)})\hat{u}_0(\xi)\right\Vert_{L^2(\mathbb{R}^n)}.
\]
Since $A(t,t_0)$ and $B(t,t_0)$ are continuous in the variable $t,$ and $u_0\in H^s,$ then the Lebesgue Dominated Convergence Theorem implies that $\lim\limits_{t\rightarrow t_1}J_1=0.$ On the other hand we have
\begin{align*}
J_2 &\leq  \left\Vert (U(t,t_0)-U(t_1,t_0))\theta\int_{t_0}^{t_1}U(t_0,\tau)\{(|x|^{-\lambda}*|u(\tau)|^2)u(\tau)\}d\tau\right\Vert_{H^s}\\
&+ \left\Vert \theta\int_{t_1}^{t}U(t,\tau)\{(|x|^{-\lambda}*|u(\tau)|^2)u(\tau)\}d\tau\right\Vert_{H^s}:=J_3+J_4.
\end{align*}
From (\ref{grup1a}), (\ref{grup2a}) and taking into account that $u(t_1)\in H^s$ we get
\begin{eqnarray}
\left\Vert \int_{t_0}^{t_1}U(t_0,\tau)\{(|x|^{-\lambda}*|u(\tau)|^2)u(\tau)\}d\tau\right\Vert_{H^s}&=&\left\Vert U(t_0,t_1)\int_{t_0}^{t_1}U(t_1,\tau)\{(|x|^{-\lambda}*|u(\tau)|^2)u(\tau)\}d\tau\right\Vert_{H^s}\nonumber\\
&=& \left\Vert \int_{t_0}^{t_1}U(t_1,\tau)\{(|x|^{-\lambda}*|u(\tau)|^2)u(\tau)\}d\tau\right\Vert_{H^s}<\infty.\nonumber
\end{eqnarray}
Therefore, analogously to the $\lim\limits_{t\rightarrow t_1}J_1=0,$ we obtain that $\lim\limits_{t\rightarrow t_1}J_3=0.$ Finally, in order to conclude (\ref{lim2}) we need to prove that $\lim\limits_{t\rightarrow t_1}J_4=0.$ For that, following the calculus in estimate (\ref{es1}) we obtain
\begin{eqnarray*}
J_4 &\leq &\left\vert \theta\int_{t_1}^{t}\Vert U(t_1,\tau)\{(|\cdot|^{-\lambda}*|u(\tau)|^2)u(\tau)\}\Vert_{H^s}d\tau\right\vert\\
&\leq & C\vert t-t_1\vert R^3\rightarrow 0,\ \mbox{as}\ t\rightarrow t_1.
\end{eqnarray*}
Thus we conclude the proof of Theorem \ref{teo1}.
\endproof
\begin{remark}\label{glo1}
As consequence of Theorem \ref{teo1} we have the local existence in $H^s(\mathbb{R}^n)$ with $s\geq \frac{\lambda}{2},$ where $\alpha,\beta$ are constants, $\beta\neq 0.$ In this case, if $s=2,$ $0<\lambda<n,$ $\theta\leq 0,$ $\alpha\geq 0,$ $\beta>0$ and $0<\lambda\leq 4,$ the existence time of solution $u,$ provided by Theorem \ref{teo1}, is $T_0=\infty.$ Indeed, in this case, the solution $u$ of (\ref{FoSch}) satisfies the following energy conservation law
\begin{equation}\label{lcx}
E(u(t))=-\beta\Vert \Delta u(t)\Vert^2_{L^2}+\alpha\Vert \nabla u(t)\Vert^2_{L^2}-\frac{\theta}{4}\int\int\frac{1}{\vert x-y\vert^{\lambda}}\vert u(t,x)\vert^2\vert u(t,y)\vert^2dxdy.
\end{equation}
Then, if $\alpha\geq 0$, $\beta>0,$ and $\theta\leq 0,$ it holds
\begin{eqnarray*}
\beta\Vert \Delta u\Vert^2_{L^2}\leq-E(u_0)+\alpha\Vert \nabla u\Vert^2_{L^2}\leq -E(u_0)+\alpha \Vert \Delta u\Vert_{L^2}\Vert \Vert u_0\Vert_{L^2},
\end{eqnarray*}
which shows that $\Vert \Delta u\Vert_{L^2}$ is globally bounded. 
\end{remark}
\subsection{Local well-posedness with piecewise constant dispersion}
Consider the integral formulation (\ref{duhamel0}). Using the decomposition (\ref{descomp}), for each $t_0$ there exists an integer $m$ such that $t_0\in I_m^1$ or $t_0\in I_m^2,$
with $I_m^1=(m,m+\tau_+]$ and $I_m^2=(m+\tau_+,m+1].$ Without loss of generality we assume that $t_0\in(0,1].$  In this case,
\[t_0\in(0,\tau_+)\cup \{\tau_+\}\cup(\tau_+,1)\cup\{1\}.\]
If $t_0\in(0,\tau_+),$ considering the function $\beta(t)=\beta^+$, following the proof of Theorem \ref{teo1}, there exists $T_0>0$ such that $[-T_0+t_0,T_0+t_0]\subset (0,\tau_+)$  and a unique solution $u\in C([-T_0+t_0,T_0+t_0];H^s(\mathbb{R}^n)),$ $s\geq \lambda/2.$  Analogously, if $t_0\in(\tau_+,1),$ considering the function $\beta(t)=-\beta^-$,  there exists $T_0>0$ such that  $[-T_0+t_0,T_0+t_0]\subset (\tau_+,1),$ and a unique solution $u\in C([-T_0+t_0,T_0+t_0];H^s(\mathbb{R}^n)),$ $s\geq \lambda/2.$ If $t_0=\tau_+,$ considering the function $\beta_1(t)=\beta^+,$  there exists $T^*_1>0$ such that $[-T^*_1+t_0,t_0]\subset (0, \tau_+]$ and a unique solution $u_1\in C([-T^*_1+t_0,t_0];H^s(\mathbb{R}^n)),$ $s\geq \lambda/2.$ On the other hand, considering the function $\beta_2(t)=-\beta^-,$ there exists $T^*_2>0$ such that $[t_0,t_0+T^*_2]\subset (\tau_+,1]$ and a unique solution $u_2\in C([t_0,t_0+T^*_2];H^s(\mathbb{R}^n)),$ $s\geq \lambda/2.$ Thus, defining $T_0=\min\{T^*_1,T^*_2\}$ and
\[
u(t)=\left\{\begin{array}{lr}
        u_1(t)\ \ \ \text{if}\  -T_0+t_0\leq t\leq t_0,\\
        u_2(t)\ \ \ \text{if}\  t_0< t\leq t_0+T_0,
        \end{array}\right.
\]
we have that $u$ solves (\ref{FoSch}) on $[-T_0+t_0,t_0+T_0]$ with $u(t_0) = u_0.$ The continuity in time of $u(t)$ follows from $\lim_{t\rightarrow 0}\Vert u(t)-u_0\Vert_{H^s} = 0$ (see the proof of (\ref{lim2})). If $t_0=1,$ the proof follows in a similar way. The previous argument works for $\alpha$ piecewise constant, and independent of the discontinuity point of $t_+$ in (\ref{alpha_beta}).
Thus, we have proved the following result. 
\begin{theo}\label{teo1c}
Let $u_0\in H^s(\mathbb{R}^n),$ $s\geq \lambda/2,$ and $0<\lambda<n,$ $n\geq 1.$ Consider $\alpha,\beta$ periodic and piecewise constant as in (\ref{alpha_beta}).  Then there exists $T_0=T_0(\Vert u_0 \Vert_{H^s})\leq T$ and a unique solution $u$ of the Cauchy problem  (\ref{FoSch}) in the class $C([-T_0+t_0,T_0+t_0];H^s(\mathbb{R}^n))$ verifying $\Vert u\Vert_{L_{T_0}^{\infty}H^s}\leq C\Vert u_0\Vert_{H^s}.$

\end{theo}

\section{Local existence in $H^s(\mathbb{R}^n)$ with $\max\{0, \frac{\lambda}{2}-2\}\leq s<\frac{\lambda}{2}$}
In this section we prove the local existence in $H^s(\mathbb{R}^n)$  with $\max\{0, \frac{\lambda}{2}-2\}\leq s<\frac{\lambda}{2}.$ As before, we assume that the variable dispersion $\alpha,\beta$ satisfies either 
$\alpha,\beta\in C([-T+t_0,T+t_0])$  with $\beta(t)\neq 0$, for all $t\in[-T+t_0,T+t_0]$ or $\alpha,\beta$ are periodic piecewise constants.
The proof is obtained through the contraction mapping argument. However, it is not easy to obtain the solution by using the contraction mapping approach only in $C([0,T];H^s(\mathbb{R}^n)).$ As usual, we use the Strichartz estimates, obtained in Section 3, and the Hardy and Hardy-Littlewood-Sobolev inequalities in order to obtain the existence of local solutions in $C([0,T];H^s(\mathbb{R}^n))\cap L_T^q(H^s_p(\mathbb{R}^n)),$ for some admissible pair $(q,p).$ Consider the mapping $\Phi_1$ defined in (\ref{fixedpoin}), and let   
\[ Y_{T,R}^s=\left\{ u\in L_T^{\infty}(H^s(\mathbb{R}^n))\cap L_T^q(H^s_p(\mathbb{R}^n)): \Vert u\Vert_{L_T^{\infty}H^s}+ \Vert u\Vert_{L_T^{q}H_p^s}\leq R \right\},\]
with metric $d(u,v)=\Vert u-v\Vert_{L_T^{\infty}H^s}+ \Vert u-v\Vert_{L_T^{q}H_p^s}$ and the admissible pair $(q,p)=\left(\frac{12}{\lambda-2s}, \frac{6n}{3n+4s-2\lambda}\right).$ 

\subsection{Local well-posedness with continuous dispersion}

\begin{theo}\label{teo2}
Let $n\geq 1,$ $0<\lambda<n,$ $u_0\in H^s(\mathbb{R}^n),$ with $\max\{0, \frac{\lambda}{2}-2\}\leq s<\frac{\lambda}{2}$ and $(q,p)$ the admissible pair $(q,p)=\left(\frac{12}{\lambda-2s}, \frac{6n}{3n+4s-2\lambda}\right).$ Consider $\alpha,\beta\in C([-T+t_0,T+t_0])$  with $\beta(t)\neq 0$, for all $t\in[-T+t_0,T+t_0].$ Then there exists $T_0=T_0(\Vert u_0 \Vert_{H^s})\leq T$ and a unique solution $u$ of (\ref{duhamel0}) in the class $C([-T_0+t_0,T_0+t_0];H^s(\mathbb{R}^n))\cap L_{T_0}^q(H^s_p(\mathbb{R}^n)).$  
\end{theo}
\proof 
Since $U(t,t_0)$ is unitary in $H^s,$ using Proposition \ref{pr1}, Lemma \ref{LeibRule}, Hardy and H\"{o}lder inequalities and Sobolev embeddings, we obtain
\begin{align*}
\|\Phi_1(u)\|_{L^{\infty}_TH^s}&\leq  \|U(t,t_0)u_0\|_{L^{\infty}_TH^s}+\left\|\theta\int_{t_0}^tU(t,\tau)\{(|x|^{-\lambda}*|u(\tau)|^{2})u(\tau)\}d\tau\right\|_{L^{\infty}_TH^s}\\
&\apprle \|u_0\|_{H^s}+\left\|(|x|^{-\lambda}*|u|^{2})u\right\|_{L^{q'}_TH_{p'}^s}\\
&\apprle \|u_0\|_{H^s}+\left\||x|^{-\lambda}*|u|^{2}\right\|_{L^{q'}_TL_x^{\frac{3n}{\lambda-2s}}}\|u\|_{L_T^{\infty}H^s} +
 \left\||x|^{-\lambda}*|u|^{2}\right\|_{L^d_TH_{\frac{3n}{2\lambda-s}}^s}\|u\|_{L_T^qL_x^b}\\
&\apprle \|u_0\|_{H^s}+\left\|u\right\|^2_{L^{2q'}_TL_x^b}\|u\|_{L_T^{\infty}H^s} +
 \left\||u|^{2}\right\|_{L^d_TH_{\frac{3n}{3n-s-\lambda}}^s}\|u\|_{L_T^qL_x^b}\\
 &\apprle \|u_0\|_{H^s}+\left\|u\right\|^2_{L^{2q'}_TL_x^b}\|u\|_{L_T^{\infty}H^s} +
 \|u\|_{L^{\infty}_TH^s}\|u\|_{L^d_TL_x^b}\|u\|_{L_T^qL_x^b}\\
&\apprle \|u_0\|_{H^s}+\left\|u\right\|^2_{L^{2q'}_TH_p^s}\|u\|_{L_T^{\infty}H^s} +
 \|u\|_{L^{\infty}_TH^s}\|u\|_{L^d_TH_p^s}\|u\|_{L_T^qH_p^s}\\
 &\apprle \|u_0\|_{H^s}+T^{\rho}\left\|u\right\|^2_{L^q_TH_p^s}\|u\|_{L_T^{\infty}H^s},
 \end{align*}
with $\rho=1+\frac{s}{2}-\frac{\lambda}{4},$ $d=\frac{6}{6-\lambda+2s}$  and  $b=\frac{6n}{3n-2s-2\lambda}.$ In the same way, we also have
\begin{align*}
\|\Phi_1(u)\|_{L_T^qH_p^s}&\leq  \|u_0\|_{H^s}+\left\|(|x|^{-\lambda}*|u|^{2})u\right\|_{L^{q'}_TH_{p'}^s}\\
&\apprle  \|u_0\|_{H^s}+T^{\rho}\left\|u\right\|^2_{L^q_TH_p^s}\|u\|_{L_T^{\infty}H^s}.\\
\end{align*}
Thus, if we choose $R$ and $T_0\leq T$ such that $C\Vert u_0\Vert_{H^s}\leq \frac R2$ and $CT_0^{\rho}R^2\leq \frac 12,$ we have that $\Phi_1$ maps $Y_{T_0,R}^s$ to itself. Now, let $u,v\in Y_{T,R}^s.$ Then, from Proposition \ref{pr1}  we get
\begin{align*}
d(\Phi_1(u),\Phi_1(v))&\apprle \left\|(|x|^{-\lambda}*|u|^{2})u-(|x|^{-\lambda}*|v|^{2})v\right\|_{L_T^{q'}H^s_{p'}}\\
&\apprle \left\|(|x|^{-\lambda}*|u|^{2})(u-v)\right\|_{L_T^{q'}H^s_{p'}}+\left\|(|x|^{-\lambda}*(|u|^{2}-|v|^{2}))v\right\|_{L_T^{q'}H^s_{p'}}.
\end{align*}
Using Lemma \ref{LeibRule}, the Hardy and H\"{o}lder inequalities, we have
\begin{align*}
\left\|(|x|^{-\lambda}*|u|^{2})(u-v)\right\|_{L_T^{q'}H^s_{p'}}&\apprle \left\||x|^{-\lambda}*|u|^{2}\right\|_{L^{q'}_TL_x^{\frac{3n}{\lambda-2s}}}\|u-v\|_{L_T^{\infty}H^s} +
\left\||x|^{-\lambda}*|u|^{2}\right\|_{L^d_TH_{\frac{3n}{2\lambda-s}}^s}\|u-v\|_{L_T^qL_x^b}\\
&\apprle \left\|u\right\|^2_{L^{2q'}_TL_x^b}\|u-v\|_{L_T^{\infty}H^s} +
 \|u\|_{L^{\infty}_TH^s}\|u\|_{L^d_TL_x^b}\|u-v\|_{L_T^qL_x^b}\\
  &\apprle T^{\rho}\left\|u\right\|^2_{L^q_TH_p^s}\|u-v\|_{L_T^{\infty}H^s} +
T^{\rho} \|u\|_{L^{\infty}_TH^s}\|u\|_{L^q_TH^s_p}\|u-v\|_{L_T^qH^s_p}\\
&\apprle T^{\rho}R^2 d(u,v).
 \end{align*}
In a similar way,
\begin{align*}
\left\|(|x|^{-\lambda}*(|u|^{2}-|v|^{2}))v\right\|_{L_T^{q'}H^s_{p'}}\apprle T^{\rho}R^2 d(u,v).
 \end{align*}
Thus, if $T_0\leq T$ is small enough, $\Phi_1$ is a contraction. Consequently, $\Phi_1$ has a unique fixed point at $Y_{T_0,R}^s$ which is solution of (\ref{duhamel0}). The time-continuity of the solution follows in the same spirit of the end of the proof of Theorem \ref{teo1}.
\endproof
\\
\subsection{Local well-posedness with piecewise constant dispersion}
Using the decomposition (\ref{descomp}), for each $t_0$ there exists an integer $m$ such that $t_0\in I_m^1$ or $t_0\in I_m^2,$
with $I_m^1=(m,m+\tau_+]$ and $I_m^2=(m+\tau_+,m+1].$ Without loss of generality we assume that $t_0\in(0,1].$  In this case,
\[t_0\in(0,\tau_+)\cup \{\tau_+\}\cup(\tau_+,1)\cup\{1\}.\]
If $t_0\in(0,\tau_+)$ or $t_0\in(\tau_+,1),$ following the proof of Theorem \ref{teo2}, there exists $T_0>0$ and a unique solution $u\in C([-T_0+t_0,T_0+t_0];H^s(\mathbb{R}^n))\cap L_{T_0}^q(H^s_p(\mathbb{R}^n)),$ with $\max\{0, \frac{\lambda}{2}-2\}\leq s<\frac{\lambda}{2}.$ Analogously, if $t_0=\tau_+,$ considering the function $\beta_1(t)=\beta^+,$ there exists $T^*_1>0$ such that $[-T^*_1+t_0,t_0]\subset (0, \tau_+]$ and a unique solution $u_1\in C([-T^*_1+t_0,t_0];H^s(\mathbb{R}^n))\cap L_{T_0}^q(H^s_p(\mathbb{R}^n)),$ with $\max\{0, \frac{\lambda}{2}-2\}\leq s<\frac{\lambda}{2}.$ On the other hand, by considering the function $\beta_2(t)=-\beta^-,$ there exists $T^*_2>0$ such that $[t_0,t_0+T^*_2]\subset (\tau_+,1]$ and a unique solution $u_2\in C([t_0,t_0+T^*_2];H^s(\mathbb{R}^n))\cap L_T^q(H^s_p(\mathbb{R}^n)),$ with $\max\{0, \frac{\lambda}{2}-2\}\leq s<\frac{\lambda}{2}.$ Thus, defining $T_0=\min\{T^*_1,T^*_2\}$ and
\[
u(t)=\left\{\begin{array}{lr}
        u_1(t)\ \ \ \text{if}\  -T_0+t_0\leq t\leq t_0,\\
        u_2(t)\ \ \ \text{if}\  t_0< t\leq t_0+T_0,
        \end{array}\right.
\]
we have that $u$ solves (\ref{FoSch}) on $[-T_0+t_0,t_0+T_0]$ with $u(t_0) = u_0.$ Thus we have the following result. 
\begin{theo}\label{teo2c}
Let $u_0\in H^s(\mathbb{R}^n)$ with $\max\{0, \frac{\lambda}{2}-2\}\leq s<\frac{\lambda}{2}$ and $0<\lambda<n,$ $n\geq 1.$ Consider $\alpha,\beta$ periodic and piecewise constant as in (\ref{alpha_beta}).  Then there exists $T_0=T_0(\Vert u_0 \Vert_{H^s})\leq T$ and a unique solution $u$ of the Cauchy problem  (\ref{FoSch}) in the class $C([-T_0+t_0,T_0+t_0];H^s(\mathbb{R}^n))\cap L_{T_0}^q(H^s_p(\mathbb{R}^n)).$
\end{theo}

\section{Global existence}
The aim of this section is to analyze the global well-posedness of  (\ref{FoSch}).  We prove that the local solution of the initial value problem (\ref{FoSch}), with initial data in $L^2$ and $H^1,$ can be extended to the real line  $\mathbb{R}.$ 

\subsection{Global existence in $L^2(\mathbb{R}^n)$}
In this subsection, we analyze the global existence of solutions for the model (\ref{FoSch}) with $\alpha,\beta$ verifying either 
$\alpha,\beta\in C(\mathbb{R})$  with $\beta(t)\neq 0$, for all $t\in\mathbb{R},$ or $\alpha,\beta$ are periodic piecewise constants. Taking into account the mass conservation $\Vert u(t)\Vert_{L^2}=\Vert u(t_0)\Vert_{L^2}$ and the local theory in $L^2,$ we are able to extend the local solution obtained in Theorem \ref{teo2} globally in time. This is the content of next theorem. 

\begin{theo}\label{Global1}
Let $u_0\in L^2(\mathbb{R}^n)$ and $0<\lambda<\mbox{min}\{n,4\}$. Then, the local solution to the initial value problem (\ref{FoSch})
obtained in Theorems \ref{teo2}, \ref{teo2c} can be extended to the real line  $\mathbb{R}.$
\end{theo}
\proof First we consider the case $\alpha,\beta$ are periodic piecewise constants.
Note that in the proof of Theorem \ref{teo2c}, the time existence of the solution $u(t)$ depends only on   $\Vert u_0\Vert_{L^2}.$ More exactly, $T$ can be taken as
\[T_0^{1-\frac{\lambda}{4}}=\frac{1}{8C^2\Vert u_0\Vert_{L^2}^2}.\]
Since $\Vert u(t)\Vert_{X_{T_0,R}^0}\leq C\Vert u_0\Vert_{L^2}$ and $\Vert u(t)\Vert_{L^2}=\Vert u_0\Vert_{L^2}$ for all $t$ on the time-interval of the existence, a standard continuity argument implies that, on each subinterval $I_m^1$ and $I_m^2,$ there exists a solution $u\in L^{\infty}(I_m^{1,2};L^2(\mathbb{R}^n)).$ Considering the union of sub-intervals $I_m,$ we infer the existence of a solution $u\in L^{\infty}(\mathbb{R};L^2(\mathbb{R}^n)).$ The continuity in time is obtained in a similar way as in Theorem \ref{teo1}, therefore there exists a solution $u\in C(\mathbb{R};L^2(\mathbb{R}^n)).$ 

In the case $\alpha,\beta\in C(\mathbb{R})$  with $\beta(t)\neq 0$, for all $t\in\mathbb{R},$ the $L^2$-conservative law and the local theory in $L^2$ provided by Theorem \ref{teo2}, also give the global existence in $L^2.$ 

\endproof
\subsection{Global existence in $H^1(\mathbb{R}^n)$}
If $\alpha$ and $\beta$ are constants, the solution $u$ of (\ref{FoSch}) satisfies the following energy conservation law
\begin{equation}\label{lc}
E(u(t))=-\beta\Vert \Delta u(t)\Vert^2_{L^2}+\alpha\Vert \nabla u(t)\Vert^2_{L^2}-\frac{\theta}{4}\int\int\frac{1}{\vert x-y\vert^{\lambda}}\vert u(t,x)\vert^2\vert u(t,y)\vert^2dxdy.
\end{equation}
Therefore, for some particular signs of $\alpha, \beta, \theta,$  and by using the following generalized Gagliardo-Niremberg inequality
\begin{equation*}
\int_{\mathbb{R}^n}(|x|^{-\lambda}*|u|^2)\vert u\vert^2dx\leq C\left(\int_{\mathbb{R}^n}\vert \nabla u\vert^2\right)^{\lambda/2}\left(\int_{\mathbb{R}^n}\vert u\vert^2\right)^{(4-\lambda)/2},
\end{equation*}
we get the {\it a priori} estimate $\Vert \nabla u\Vert_{L^2(\mathbb{R}^n)}\leq C,$ which implies  
the existence of global solution in $H^1.$ Unfortunately, if $\alpha,\beta$ are not constants, (\ref{lc}) does not hold. However, we are able to obtain existence of global solution in $H^1$ by combining the  $L^2$-conservative law, the local well-posedness in $H^1$ and an argument of blow up alternative. 

\begin{theo}\label{Global1H1}
Let $u_0\in H^1(\mathbb{R}^n), 0<\lambda<n$ and $\lambda< 4.$ Assume that $\alpha,\beta\in C(\mathbb{R})$  with $\beta(t)\neq 0$, for all $t\in\mathbb{R}.$ Then, the local solution to the initial value problem (\ref{FoSch}) can be extended to $\mathbb{R}.$
\end{theo}
\proof Recall that the $L^2$-solutions $u$ of  (\ref{FoSch})-(\ref{alpha_beta}) satisfies the mass conservation law 
\[\Vert u(t)\Vert_{L^2(\mathbb{R}^n)}=\Vert u_0\Vert_{L^2(\mathbb{R}^n)}.\]
Moreover, the following relation holds
\begin{equation}\label{RelImp1}
\partial_t\Vert \nabla u(t)\Vert_{L^2(\mathbb{R}^n)}^2=-2\theta \mbox{Im}\int_{\mathbb{R}^n}\nabla [|u(x,t)|^{2}u(x,t)]\nabla\overline{u}(x,t)dx.
\end{equation}
Then, if $p\geq 1$ is such that $\frac{2}{p}+\frac{\lambda}{n}=1,$ from (\ref{RelImp1}), Hardy-Littlewood-Sobolev and H\"older inequalities, we get
\begin{align}
\partial_t\Vert \nabla u(t)\Vert_{L^2(\mathbb{R}^n)}^2 &\leq 2\vert \theta\vert \left\vert  \int_{\mathbb{R}^n}\int_{\mathbb{R}^n}\vert x-y\vert^{-\lambda}\nabla\vert u(y)\vert^2u(x)\nabla\bar{u}(x)dydx\right\vert\nonumber\\
&\leq  C\Vert\nabla\vert u\vert^2\Vert_{L^{\frac{2p}{p+2}}(\mathbb{R}^n)}\Vert u \nabla\bar{u}\Vert_{L^{\frac{2p}{p+2}}(\mathbb{R}^n)}\nonumber\\
&\leq  C\Vert\nabla u (t)\Vert^2_{L^2(\mathbb{R}^n)}\Vert u(t)\Vert^2_{L^p(\mathbb{R}^n)}.\label{glo1}
\end{align}
Suppose that $T_{\mbox{\tiny{max}}}<\infty.$ Then, if $\frac{1}{q_1}+\frac{1}{q_2}=\frac{1}{2},$ from (\ref{glo1}) and Gronwall inequality, for $0\leq t<T_{\mbox{\tiny{max}}}$ we have that
\begin{align}\label{GradAprioria}
\Vert \nabla u(t)\Vert_{L^2(\mathbb{R}^n)}^2 &\leq \Vert \nabla u_0\Vert_{L^2(\mathbb{R}^n)}^2\exp(C\int_0^t\Vert u(\tau)\Vert^2_{L^p(\mathbb{R}^n)}d\tau)\nonumber\\
&\leq \Vert \nabla u_0\Vert_{L^2(\mathbb{R}^n)}^2\exp(CT_{\mbox{\tiny{max}}}^{\frac{1}{q_1}}\Vert u\Vert^2_{L^{2q_2}([0,T_{\mbox{\tiny{max}}}];L^p(\mathbb{R}^n)}).
\end{align}
Without loss of generality we consider $t_0=0;$ then, we use the equation
\[u(t)=U(t)u_0+i\theta\int_{0}^{t}U(t,\tau)\{(|x|^{-\lambda}*|u(\tau)|^2)u(\tau)\}d\tau,
\]
in order to obtain an estimate of $\Vert u\Vert_{L^{2q_2}([0,T_{\mbox{\tiny{max}}}];L^p(\mathbb{R}^n))}.$ Indeed, for $0\leq t\leq \tilde{T}_0\leq T_{\mbox{\tiny{max}}},$ we get
\begin{align*}
\Vert u\Vert_{L^{2q_2}([0,\tilde{T}_0];L^p(\mathbb{R}^n))}&\leq \|U(t,0)u_0\|_{L^{2q_2}([0,\tilde{T}_0];L^p(\mathbb{R}^n))}\nonumber\\
&+C\int_{0}^{\tilde{T}_0}\|U(t,\tau)\{(|x|^{-\lambda}*|u(\tau)|^{2})u(\tau)\}
\|_{L^{2q_2}([0,\tilde{T}_0];L^p(\mathbb{R}^n))}d\tau.
\end{align*}
At this point we need to consider that $(2q_2,p)$ is an admissible pair. This condition implies that $q_2=\frac{4}{\lambda}$ and thus, we find the restriction $\lambda\leq 4.$ Therefore, from H\"older inequality and Lemma \ref{HLSineq}, we obtain
\begin{align}\label{IneUL1a}
\| u\|_{L^{\frac{8}{\lambda}}([0,\tilde{T}_0];L^p(\mathbb{R}^n))}&\leq C\|u_0\|_{L^2(\mathbb{R}^n)}+C\int_{0}^{\tilde{T}_0}\|(|x|^{-\lambda}*|u(\tau)|^{2})u(\tau)
\|_{L^2(\mathbb{R}^n)}d\tau\notag\\
 & \leq C\|u_0\|_{L^2(\mathbb{R}^n)}+C\int_{0}^{\tilde{T}_0}\|u(\tau)\|_{L^{p}(\mathbb{R}^n)}\||x|^{-\lambda}*|u(\tau)|^{2}
\|_{L^{\frac{2p}{p-2}}(\mathbb{R}^n)}d\tau\notag\\
&= C\|u_0\|_{L^2(\mathbb{R}^n)}+C\|u_0\|_{L^2(\mathbb{R}^n)}\int_{0}^{\tilde{T}_0}\|u(\tau)\|_{L^{p}(\mathbb{R}^n)}^{2}d\tau\notag\\
& \leq C\|u_0\|_{L^2(\mathbb{R}^n)}+C\|u_0\|_{L^2(\mathbb{R}^n)}\tilde{T}^{\frac{4-\lambda}{4}}_0\|u\|_{L^{\frac{8}{\lambda}}([0,\tilde{T}_0];L^p(\mathbb{R}^n))}^{2}.
\end{align}
We claim that
\[
\| u\|_{L^{\frac{8}{\lambda}}([0,\tilde{T}_0];L^p(\mathbb{R}^n))}\leq 2C \|u_0\|_{L^2(\mathbb{R}^n)},
\]
provided $0<\tilde{T}_0^{\frac{4-\lambda}{4}}<\frac{1}{4C^2\|u_0\|^2_{L^2(\mathbb{R}^n)}}.$ Assume by contradiction that $
\| u\|_{L^{\frac{8}{\lambda}}([0,\tilde{T}_0];L^p(\mathbb{R}^n))} >2C \|u_0\|_{L^2(\mathbb{R}^n)}.$
By continuity there exists $T^*\leq \tilde{T}_0$ such that
\begin{equation}\label{IquaU1a}
\| u\|_{L^{\frac{8}{\lambda}}([0,T^*];L^p(\mathbb{R}^n))}= 2C\|u_0\|_{L^2(\mathbb{R}^n)}.
\end{equation}
Notice that (\ref{IneUL1a}) is also valid for $T^*$ instead of $\tilde{T}_0$. Replacing (\ref{IquaU1a}) in (\ref{IneUL1a}) we get that
\[
2C\|u_0\|_{L^2(\mathbb{R}^n)}\leq C\|u_0\|_{L^2(\mathbb{R}^n)}+C\|u_0\|_{L^2(\mathbb{R}^n)}(T^*)^{\frac{4-\lambda}{4}}4C^2\|u_0\|^2_{L^2(\mathbb{R}^n)},
\]
which implies $1\leq (T^*)^{\frac{4-\lambda}{4}}4C^2\|u_0\|^2_{L^2(\mathbb{R}^n)}\leq  \tilde{T}^{\frac{4-\lambda}{4}}_04C^2\|u_0\|^2_{L^2(\mathbb{R}^n)}$ and this contradicts the choice of $\tilde{T}_0.$ Therefore, considering ${T}_0^{\frac{4-\lambda}{4}}=\frac{1}{8C^2\|u_0\|^2_{L^2(\mathbb{R}^n)}}$ we have
\begin{equation}\label{APriori1a}
\| u\|_{L^{\frac{8}{\lambda}}([0,{T}_0];L^p(\mathbb{R}^n))}\leq 2C \|u_0\|_{L^2(\mathbb{R}^n)}.
\end{equation}
If $T_0=T_{\mbox{\tiny{max}}}$ we finish the proof. Suppose that $T_0<T_{\mbox{\tiny{max}}}.$ Then, we repeat the above argument to obtain {\it a priori} estimate in the interval $[0,2T_0].$ Indeed, from Duhamel's formula we have that
\[u(t)=U(t,T_0)u_0(T_0)+i\theta\int_{T_0}^tU(t,\tau)\{(|x|^{-\lambda}*|u(\tau)|^{2})u(\tau)\}d\tau.\]
For $T_0\leq t\leq T_0+\tilde{T}_1<T_{\mbox{\tiny{max}}},$ from the H\"older inequality and Lemma \ref{HLSineq} 
 we arrived at
\begin{align*}
\| u\|_{L^{\frac{8}{\lambda}}([T_0,T_0+\tilde{T}_1];L^p(\mathbb{R}^n))}&\leq C\|u_0\|_{L^2(\mathbb{R}^n)}+C\int_{T_0}^{t}\|(|x|^{-\lambda}*|u(\tau)|^{2})u(\tau)
\|_{L^2(\mathbb{R}^n)}d\tau\notag\\
 & \leq C\|u_0\|_{L^2(\mathbb{R}^n)}+C\int_{T_0}^{T_0+\tilde{T}_1}\|u(\tau)\|_{L^{p}(\mathbb{R}^n)}^{2}\|u(\tau)
\|_{L^2(\mathbb{R}^n)}d\tau\notag\\
& \leq C\|u_0\|_{L^2(\mathbb{R}^n)}+C\|u_0\|_{L^2(\mathbb{R}^n)}\tilde{T}_1^{\frac{4-\lambda}{4}}\|u\|_{L^{\frac{8}{\lambda}}([T_0,T_0+\tilde{T}_1];L^p(\mathbb{R}^n))}^{2}.
\end{align*}
Again, taking $0<\tilde{T}_1^{\frac{4-\lambda}{4}}<\frac{1}{4C^2\|u_0\|^2_{L^2(\mathbb{R}^n)}},$ we obtain
\begin{equation}\label{APriori2f}
\| u\|_{L^{\frac{8}{\lambda}}([T_0,T_0+\tilde{T}_1];L^p(\mathbb{R}^n))}\leq 2C \|u_0\|^2_{L^2(\mathbb{R}^n)}.
\end{equation}
Therefore, we can chose $\tilde{T}_1=T_0.$ From (\ref{APriori1a}) and (\ref{APriori2f}), we obtain
\[
\| u\|_{L^{\frac{8}{\lambda}}([0,2T_0];L^p(\mathbb{R}^n))}\leq 4C \|u_0\|^2_{L^2(\mathbb{R}^n)}.
\]
Repeating this process a finite number of steps and using the value of $T_0$ we arrived at,
\begin{equation}\label{Apriori3a}
\| u\|_{L^{\frac{8}{\lambda}}([0,T_{\mbox{\tiny{max}}});L^p(\mathbb{R}^n))}\leq C \frac{T_{\mbox{\tiny{max}}}}{T_0}\|u_0\|^2_{L^2(\mathbb{R}^n)}\leq CT_{\mbox{\tiny{max}}}\|u_0\|^2_{L^2(\mathbb{R}^n)}.
\end{equation}
Replacing (\ref{Apriori3a}) in (\ref{GradAprioria}) we get the estimate
\[
\Vert \nabla u(t)\Vert_{L^2(\mathbb{R}^n)}^2\leq \Vert \nabla u_0\Vert_{L^2(\mathbb{R}^n)}^2\exp(CT_{\mbox{\tiny{max}}}^{\frac{4-\lambda}{4}}\|u_0\|^2_{L^2(\mathbb{R}^n)}),
\]
for any $0\leq t<T_{\mbox{\tiny{max}}},$ which is a contradiction to the blow-up alternative. Therefore, $T_{\mbox{\tiny{max}}}=\infty.$

\endproof
\begin{remark}
Combining the arguments in the proof of Theorem \ref{teo1c} with those in the proof of Theorem \ref{Global1H1} we can prove that the local solution to the initial value problem (\ref{FoSch})
obtained in the case $\alpha,\beta$ piecewise constant, can be extended to  $\mathbb{R}.$ 
\end{remark}
\begin{remark}\label{remlwp}
We could try to prove the global existence in $H^s$ combining the local existence in $H^k,$ $k\in \mathbb{N},$ and an interpolation argument. We could use an induction argument on $k$ to prove global existence for initial data in $H^k(\mathbb{R}^n)$ with $k\geq 2$ an integer. For that, we need an {\it a priori} estimate to show that the global existence of (\ref{FoSch}) in $H^{k-1}(\mathbb{R}^n)$ implies the global existence in $H^{k}(\mathbb{R}^n).$ However, if we multiply the first equation in (\ref{FoSch}) by $D_x^{2\alpha}\bar{u},$ where $\alpha$ is a multi-index with $|\alpha|\leq k,$ $k>1,$ next, conjugate (\ref{FoSch}) and multiply it by $D_x^{2\alpha}u,$ and then, we add the two obtained equations and use basic properties of the Laplacian and the operator $\Delta^2,$ we arrived at
\begin{equation}\label{PartImag1z}
\partial_t\Vert D_x^{\alpha} u(t)\Vert_{L_x^2(\mathbb{R}^n)}^2=-2\theta \mbox{Im}\int_{\mathbb{R}^n}D_x^{\alpha}[(|x|^{-\lambda}*|u|^{2})u]D_x^{\alpha}\overline{u}dx.
\end{equation}
By Leibnitz's rule we have that
\begin{eqnarray}
D^{\alpha}_x[(|x|^{-\lambda}*|u|^{2})u]&=&(|x|^{-\lambda}*|u|^{2})D^\alpha_xu+D^\alpha_x(|x|^{-\lambda}*|u|^{2})u\nonumber\\
&&+\sum\limits_{0<\beta< \alpha}\binom{\alpha}{\beta}D^{\beta}_x(|x|^{-\lambda}*|u|^{2})D^{\alpha-\beta}_xu\nonumber\\
&=&(|x|^{-\lambda}*|u|^{2})D^\alpha_xu+(|x|^{-\lambda}*D^\alpha_x(|u|^{2}))u\nonumber\\
&&+\sum\limits_{0<\beta< \alpha}\binom{\alpha}{\beta}(|x|^{-\lambda}*D^{\beta}_x(|u|^{2}))D^{\alpha-\beta}_xu.\nonumber
\end{eqnarray}
%
%
Thus, from (\ref{PartImag1z}) we obtain
\begin{align}
\partial_t\Vert D_x^{\alpha} u\Vert_{L_x^2(\mathbb{R}^n)}^2=&-2\theta \mbox{Im}\left( \displaystyle\int_{\mathbb{R}^n}(|x|^{-\lambda}*|u|^{2})D^\alpha_xuD^\alpha_x\overline{u}dx+\int_{\mathbb{R}^n}(|x|^{-\lambda}*D^\alpha_x(|u|^{2}))uD^\alpha_x\overline{u}dx\right)\nonumber\\
&-2\theta \mbox{Im}\left(\displaystyle\sum\limits_{0<\beta< \alpha}\binom{\alpha}{\beta}\displaystyle\int_{\mathbb{R}^n}(|x|^{-\lambda}*D^{\beta}_x(|u|^{2}))D^{\alpha-\beta}_xuD_x^{\alpha}\overline{u}dx\right).\label{deriz}
\end{align}
Unfortunately, seems so difficult to control the right hand side of (\ref{deriz}) in terms of the norms $\|u\|_{H^{1}}$ and $\|u\|_{H^{k-1}}.$
\end{remark}
\subsection{Global well-posedness in $H^s(\mathbb{R}^n)$ with $s>0$ and nonlinearity $\vert u\vert^2u$}
Taking into account the Remark \ref{remlwp}, throughout this section we consider the model
\begin{equation}\label{FoSchb}
\left\{
\begin{array}{lc}
i\partial _{t}u+\alpha(t)\Delta u+\beta(t) \Delta^2  u+\theta|u|^{2}u=0, & x\in \mathbb{R}^{n},\ \ t\in \mathbb{R}, \\
u(x,t_0)=u_{0}(x), &  x\in \mathbb{R}^{n}, \ \ t_0\in\mathbb{R}.
\end{array}
\right.
\end{equation}
In this case, from Duhamel's formula we have,
\begin{equation}\label{duhamelu2u}
u(t)=U(t,t_0)u_0+i\theta\int_{t_0}^tU(t,\tau)|u(\tau)|^{2}u(\tau)d\tau.
\end{equation}
The proof of the next theorem is similar to that one of Theorem \ref{teo2}.
\begin{theo}\label{Theou2u}
Let $n\geq 1,$ $u_0\in H^s(\mathbb{R}^n),$ with $\max\{0, \frac{n}{2}-2\}\leq s<\frac{n}{2}$ and $(q,p)$ the admissible pair $(q,p)=\left(\frac{12}{n-2s}, \frac{6n}{n+4s}\right).$ Consider $\alpha,\beta\in C([-T+t_0,T+t_0])$  with $\beta(t)\neq 0$, for all $t\in[-T+t_0,T+t_0].$ Then there exists $T_0=T_0(\Vert u_0 \Vert_{H^s})\leq T$ and a unique solution $u$ of the integral equation (\ref{duhamelu2u}) in the class $C([-T_0+t_0,T_0+t_0];H^s(\mathbb{R}^n))\cap L_{T_0}^q(H^s_p(\mathbb{R}^n)).$  
\end{theo}
\proof Consider the mapping
\[
\Phi_2(u)(t)=U(t,t_0)u_0+i\theta\int_{t_0}^tU(t,\tau)|u(\tau)|^{2})u(\tau)d\tau.
\]
Since $U(t,t_0)$ is unitary in $H^s,$ using Propositions \ref{pr1}, Lemma \ref{LeibRule}, H\"{o}lder inequality and Sobolev embeddings, we obtain
\begin{align*}
\|\Phi_2(u)\|_{L^{\infty}_TH^s}&\leq  \|U(t,t_0)u_0\|_{L^{\infty}_TH^s}+\left\|\theta\int_{t_0}^tU(t,\tau)|u(\tau)|^{2}u(\tau)d\tau\right\|_{L^{\infty}_TH^s}\\
&\apprle \|u_0\|_{H^s}+\left\||u|^{2}u\right\|_{L^{q'}_TH_{p'}^s}\\
&\apprle \|u_0\|_{H^s}+\left\||u|^{2}\right\|_{L^{q'}_TL_x^{\frac{3n}{n-2s}}}\|u\|_{L_T^{\infty}H^s} +
 \left\||u|^{2}\right\|_{L^d_TH_{\frac{3n}{2n-s}}^s}\|u\|_{L_T^qL_x^b}\\
&\apprle \|u_0\|_{H^s}+\left\|u\right\|^2_{L^{2q'}_TL_x^b}\|u\|_{L_T^{\infty}H^s} +
 \left\||u|^{2}\right\|_{L^d_TH_{\frac{3n}{2n-s}}^s}\|u\|_{L_T^qL_x^b}\\
 &\apprle \|u_0\|_{H^s}+\left\|u\right\|^2_{L^{2q'}_TL_x^b}\|u\|_{L_T^{\infty}H^s} +
 \|u\|_{L^{\infty}_TH^s}\|u\|_{L^d_TL_x^b}\|u\|_{L_T^qL_x^b}\\
&\apprle \|u_0\|_{H^s}+\left\|u\right\|^2_{L^{2q'}_TH_p^s}\|u\|_{L_T^{\infty}H^s} +
 \|u\|_{L^{\infty}_TH^s}\|u\|_{L^d_TH_p^s}\|u\|_{L_T^qH_p^s}\\
 &\apprle \|u_0\|_{H^s}+T^{\rho}\left\|u\right\|^2_{L^q_TH_p^s}\|u\|_{L_T^{\infty}H^s},
 \end{align*}
with $\rho=1+\frac{s}{2}-\frac{n}{4},$ $d=\frac{6}{6-n+2s}$  and  $b=\frac{6n}{n-2s}.$ The rest of the proof es very similar to that one in Theorem \ref{teo2}.
\endproof

Next, we will analyze the global well-posedness in $H^s(\mathbb{R}^n)$ with $s\geq0.$ For that, next lemma will be useful. 
\begin{lemm}\label{BilTri1}
\cite{Bona} Let $f,g\in H^r(\mathbb{R}^n),$ with $r>\frac{1}{2}$ and $h\in H^s(\mathbb{R}^n),$ with $0\leq s\leq r.$ Then
\[\|fh\|_{H^s}\leq C\|f\|_{H^r}\|h\|_{H^s}.\]
\end{lemm}
\begin{theo}\label{Global1Hs}
Let $u_0\in H^s(\mathbb{R}^n),$  with $s\geq 0,$  $n< 4.$  Assume that $\alpha,\beta\in C(\mathbb{R})$  with $\beta(t)\neq 0$, for all $t\in\mathbb{R}.$ Then the local solution to the initial value problem (\ref{FoSchb})
can be extended to $\mathbb{R}.$
\end{theo}
\proof We already to known that the $L^2$-solutions $u$ of  (\ref{FoSchb})  also satisfies the mass conservation law
\begin{equation*}
\Vert u(t)\Vert_{L^2(\mathbb{R}^n)}=\Vert u_0\Vert_{L^2(\mathbb{R}^n)}.
\end{equation*}
Moreover, the following relation holds
\begin{equation}\label{id2}
\partial_t\Vert \nabla u(t)\Vert_{L^2(\mathbb{R}^n)}^2=-2\theta \mbox{Im}\int_{\mathbb{R}^n}\nabla [|u(x,t)|^{2}u(x,t)]\nabla\overline{u}(x,t)dx.
\end{equation}
Let $u_0\in H^1(\mathbb{R}^n)$ and $T_{\mbox{\tiny{max}}}$ be the maximal existence time of the solution to (\ref{FoSchb}). Suppose that $T_{\mbox{\tiny{max}}}<\infty.$ Then, for $0<t<T_{\mbox{\tiny{max}}},$ from the (\ref{id2}) and the H\"older inequality we arrive at
\[
\partial_t\Vert \nabla u(t)\Vert_{L^2(\mathbb{R}^n)}^2 \leq C\Vert u (t)\Vert^2_{L^\infty(\mathbb{R}^n)}\Vert\nabla u(t) \Vert^2_{L^2(\mathbb{R}^n)}.
\]
Thus, if $\frac{1}{q_1}+\frac{1}{q_2}=\frac{1}{2}$ we have
\begin{align}\label{GradApriori}
\Vert \nabla u(t)\Vert_{L^2(\mathbb{R}^n)}^2 &\leq \Vert \nabla u_0\Vert_{L^2(\mathbb{R}^n)}^2\exp(C\int_0^t\Vert u(\tau)\Vert^2_{L^\infty(\mathbb{R}^n)}d\tau)\nonumber\\
&\leq \Vert \nabla u_0\Vert_{L^2(\mathbb{R}^n)}^2\exp(CT_{\mbox{\tiny{max}}}^{\frac{1}{q_1}}\Vert u\Vert^2_{L^{2q_2}([0,T_{\mbox{\tiny{max}}}];L^\infty(\mathbb{R}^n)}).
\end{align}
Without loss of generality we consider $t_0=0;$ then, we use the equation
\[u(t)=U(t)u_0+i\theta\int_{0}^tU(t,\tau)\{|u(\tau)|^{2}u(\tau)\}d\tau,\]
in order to obtain an estimate of $\Vert u\Vert_{L^{2q_2}([0,T_{\mbox{\tiny{max}}}];L^\infty(\mathbb{R}^n))}.$ Indeed,
\[
\Vert u\Vert_{L^{2q_2}([0,\tilde{T}_0];L^\infty(\mathbb{R}^n))}\leq \|U(t,0)u_0\|_{L^{2q_2}([0,\tilde{T}_0];L^\infty(\mathbb{R}^n))}+C\int_{0}^{\tilde{T}_0}\|U(t,\tau)\{|u(\tau)|^{2}u(\tau)\}
\|_{L^{2q_2}([0,\tilde{T}_0];L^\infty(\mathbb{R}^n))}d\tau.
\]
At this point we need to use the inequality (\ref{IneLinCont}), which implies that $q_2=\frac{4}{n}$ and $n< 4.$ Hence,
\begin{align}\label{IneUL1}
\| u\|_{L^{\frac{8}{n}}([0,\tilde{T}_0];L^\infty(\mathbb{R}^n))}&\leq C\|u_0\|_{L^2(\mathbb{R}^n)}+C\int_{0}^{\tilde{T}_0}\||u(\tau)|^{2}u(\tau)
\|_{L^2(\mathbb{R}^n)}d\tau\notag\\
 & \leq C\|u_0\|_{L^2(\mathbb{R}^n)}+C\int_{0}^{\tilde{T}_0}\|u(\tau)\|_{L^{\infty}(\mathbb{R}^n)}^{2}\|u(\tau)
\|_{L^2(\mathbb{R}^n)}d\tau\notag\\
&= C\|u_0\|_{L^2(\mathbb{R}^n)}+C\|u_0\|_{L^2(\mathbb{R}^n)}\int_{0}^{\tilde{T}_0}\|u(\tau)\|_{L^{\infty}(\mathbb{R}^n)}^{2}d\tau\notag\\
& \leq C\|u_0\|_{L^2(\mathbb{R}^n)}+C\|u_0\|_{L^2(\mathbb{R}^n)}\tilde{T}^{\frac{4-n}{4}}_0\|u\|_{L^{\frac{8}{n}}([0,\tilde{T}_0];L^\infty(\mathbb{R}^n))}^{2}.
\end{align}
Now, for  $n=1,2,3,$  we claim that
\begin{equation}\label{APriori1}
\| u\|_{L^{\frac{8}{n}}([0,\tilde{T}_0];L^\infty(\mathbb{R}^n))}\leq 2C \|u_0\|_{L^2(\mathbb{R}^n)},
\end{equation}
provided $0<\tilde{T}_0^{\frac{4-n}{4}}<\frac{1}{4C^2\|u_0\|^2_{L^2(\mathbb{R}^n)}}.$ Assume by contradiction that, $
\| u\|_{L^{\frac{8}{n}}([0,\tilde{T}_0];L^\infty(\mathbb{R}^n))} >2C \|u_0\|_{L^2(\mathbb{R}^n)}.$
By continuity there exists $T^*\leq \tilde{T}_0$ such that
\begin{equation}\label{IquaU1}
\| u\|_{L^{\frac{8}{n}}([0,T^*];L^\infty(\mathbb{R}^n))}= 2C\|u_0\|_{L^2(\mathbb{R}^n)}.
\end{equation}
Notice that (\ref{IneUL1}) is also valid for $T^*$ instead of $\tilde{T}_0$. Replacing (\ref{IquaU1}) in (\ref{IneUL1}) we get that
\[
2C\|u_0\|_{L^2(\mathbb{R}^n)}\leq C\|u_0\|_{L^2(\mathbb{R}^n)}+C\|u_0\|_{L^2(\mathbb{R}^n)}(T^*)^{\frac{4-n}{4}}4C^2\|u_0\|^2_{L^2(\mathbb{R}^n)},
\]
which implies $1\leq (T^*)^{\frac{4-n}{4}}4C^2\|u_0\|^2_{L^2(\mathbb{R}^n)}\leq  \tilde{T}^{\frac{4-n}{4}}_04C^2\|u_0\|^2_{L^2(\mathbb{R}^n)}$ and this contradicts the choice of $\tilde{T}_0.$ Therefore, considering ${T}_0^{\frac{4-n}{4}}=\frac{1}{8C^2\|u_0\|^2_{L^2(\mathbb{R}^n)}}$ we have
\begin{equation}\label{APriori1f}
\| u\|_{L^{\frac{8}{n}}([0,{T}_0];L^\infty(\mathbb{R}^n))}\leq 2C \|u_0\|_{L^2(\mathbb{R}^n)}.
\end{equation}
If $T_0=T_{\mbox{\tiny{max}}}$ we finish the proof. Suppose that $T_0<T_{\mbox{\tiny{max}}}.$ Then, we repeat the above argument to obtain {\it a priori} estimate in the interval $[0,2T_0].$ Indeed, from Duhamel's formula we have that
\[u(t)=U(t,T_0)u_0(T_0)+i\theta\int_{T_0}^tU(t,\tau)\{|u(\tau)|^{2}u(\tau)\}d\tau.\]
For $T_0\leq t\leq T_0+\tilde{T}_1<T_{\mbox{\tiny{max}}},$ from (\ref{IneLinCont}), we arrived at

\begin{align*}
\| u\|_{L^{\frac{8}{n}}([T_0,T_0+\tilde{T}_1];L^\infty(\mathbb{R}^n))}&\leq C\|u_0\|_{L^2(\mathbb{R}^n)}+C\int_{T_0}^{t}\||u(\tau)|^{2}u(\tau)
\|_{L^2(\mathbb{R}^n)}d\tau\notag\\
 & \leq C\|u_0\|_{L^2(\mathbb{R}^n)}+C\int_{T_0}^{T_0+\tilde{T}_1}\|u(\tau)\|_{L^{\infty}(\mathbb{R}^n)}^{2}\|u(\tau)
\|_{L^2(\mathbb{R}^n)}d\tau\notag\\
& \leq C\|u_0\|_{L^2(\mathbb{R}^n)}+C\|u_0\|_{L^2(\mathbb{R}^n)}\tilde{T}_1^{\frac{4-n}{4}}\|u\|_{L^{\frac{8}{n}}([T_0,T_0+\tilde{T}_1];L^\infty(\mathbb{R}^n))}^{2}.
\end{align*}
Again, taking $0<\tilde{T}_1^{\frac{4-n}{4}}<\frac{1}{4C^2\|u_0\|^2_{L^2(\mathbb{R}^n)}},$ we obtain
\begin{equation}\label{APriori2}
\| u\|_{L^{\frac{8}{n}}([T_0,T_0+\tilde{T}_1];L^\infty(\mathbb{R}^n))}\leq 2C \|u_0\|^2_{L^2(\mathbb{R}^n)}.
\end{equation}
Therefore, we can chose $\tilde{T}_1=T_0.$ From (\ref{APriori1}) and (\ref{APriori2}), we obtain
\[
\| u\|_{L^{\frac{8}{n}}([0,2T_0];L^\infty(\mathbb{R}^n))}\leq 4C \|u_0\|^2_{L^2(\mathbb{R}^n)}.
\]
Repeating this process a finite number of steps and using the value of $T_0$ we arrived at,
\begin{equation}\label{Apriori3}
\| u\|_{L^{\frac{8}{n}}([0,T_{\mbox{\tiny{max}}});L^\infty(\mathbb{R}^n))}\leq C \frac{T_{\mbox{\tiny{max}}}}{T_0}\|u_0\|^2_{L^2(\mathbb{R}^n)}\leq CT_{\mbox{\tiny{max}}}\|u_0\|^2_{L^2(\mathbb{R}^n)}.
\end{equation}
Replacing (\ref{Apriori3}) in (\ref{GradApriori}) we get the {\it a priori} estimate,
\[
\Vert \nabla u(t)\Vert_{L^2(\mathbb{R}^n)}^2\leq \Vert \nabla u_0\Vert_{L^2(\mathbb{R}^n)}^2\exp(CT_{\mbox{\tiny{max}}}^{\frac{4-n}{4}}\|u_0\|^2_{L^2(\mathbb{R}^n)}),
\]
for any $0\leq t<T_{\mbox{\tiny{max}}},$ which is a contradiction to the blow-up alternative. Therefore, $T_{\mbox{\tiny{max}}}=\infty.$

Next, we use an induction argument on $k$ to prove global well-posedness for initial data in $H^k(\mathbb{R}^n)$ with $k\geq 2$ an integer. For this we use an {\it a priori} estimate to show that the global well-posedness of (\ref{FoSchb}) in $H^{k-1}(\mathbb{R}^n)$ implies the global well-posedness in $H^{k-1}(\mathbb{R}^n).$

First, multiply equation (\ref{FoSchb}) by $D_x^{2\alpha}\bar{u},$ where $\alpha$ is a multi-index with $|\alpha|\leq k,$ next conjugate (\ref{FoSchb}) and multiply by $D_x^{2\alpha}u,$ add the two equations obtained and from the properties of the Laplacian and the operator $\Delta^2,$ we arrived at

\begin{equation}\label{PartImag1}
\partial_t\Vert D_x^{\alpha} u(t)\Vert_{L_x^2(\mathbb{R}^n)}^2=-2\theta \mbox{Im}\int_{\mathbb{R}^n}D_x^{\alpha}[|u(x,t)|^{2}u(x,t)]D_x^{\alpha}\overline{u}(x,t)dx.
\end{equation}
By Leibnitz's rule we have that
\[D^{\alpha}_x(u^2\overline{u})=\sum\limits_{\beta\leq \alpha}\binom{\alpha}{\beta}D^{\beta}_x(u^2)D^{\alpha-\beta}_x\overline{u}=\overline{u} D^{\alpha}_x(u^2)+ u^2D^{\alpha}_x(\overline{u})+\sum\limits_{0<\beta< \alpha}\binom{\alpha}{\beta}D^{\beta}_x(u^2)D^{\alpha-\beta}_x\overline{u}\]
and
\[\overline{u}D^{\alpha}_x(u^2)=\overline{u}\sum\limits_{\beta\leq \alpha}\binom{\alpha}{\beta}D^{\beta}_x(u)D^{\alpha-\beta}_x u=2|u|^2 D^{\alpha}_x(u)+\overline{u}\sum\limits_{0<\beta< \alpha}\binom{\alpha}{\beta}D^{\beta}_x(u)D^{\alpha-\beta}_x u.\]
Now, from (\ref{PartImag1}) we obtain
\begin{align*}
\partial_t\Vert D_x^{\alpha} u(t)\Vert_{L_x^2(\mathbb{R}^n)}^2=&-2\theta \mbox{Im}\left(\sum\limits_{0<\beta< \alpha}\binom{\alpha}{\beta}\int_{\mathbb{R}^n}\overline{u}D^{\beta}_x u D^{\alpha-\beta}_x uD_x^{\alpha}\overline{u}dx+\int_{\mathbb{R}^n} u^2D^{\alpha}_x \overline{u} D^{\alpha}_x \overline{u}dx\right)\\
&-2\theta \mbox{Im}\left(\sum\limits_{0<\beta< \alpha}\binom{\alpha}{\beta}\int_{\mathbb{R}^n}D^{\beta}_x (u^2) D^{\alpha-\beta}_x \overline{u}D_x^{\alpha}\overline{u}dx\right).
\end{align*}
For $0<\beta<\alpha,$ using Proposition \ref{BilTri1} with $s=0$ we get 
\begin{align}\label{IntIne1}
\left|\int_{\mathbb{R}^n}\overline{u}D^{\beta}_x u D^{\alpha-\beta}_x uD_x^{\alpha}\overline{u}dx\right|&\leq \|D_x^{\alpha}\overline{u}\|_{L^2(\mathbb{R}^n)}\|\overline{u}D^{\beta}_x u D^{\alpha-\beta}_x u\|_{L^2(\mathbb{R}^n)}\notag\\
&\leq C\|D_x^{\alpha}\overline{u}\|_{L^2(\mathbb{R}^n)}\|u\|_{H^{\frac{1}{2}+}}\|D^{\beta}_x u\|_{H^{\frac{1}{2}+}}\|D^{\alpha-\beta}_x u\|_{L^2(\mathbb{R}^n)}\notag\\
&\leq C\|D_x^{\alpha}\overline{u}\|^2_{L^2(\mathbb{R}^n)}\|u\|_{H^{1}}\|u\|_{H^{k-1}}
\leq C\|u\|^2_{H^{k-1}}\|D^{\alpha}u\|^2_{L^2(\mathbb{R}^n)}
\end{align}
For $0<\beta<\alpha,$ using Proposition \ref{BilTri1} with $s=0$ we get
\begin{align}\label{IntIne2}
\left|\int_{\mathbb{R}^n}D^{\beta}_x (u^2) D^{\alpha-\beta}_x \overline{u}D_x^{\alpha}\overline{u}dx\right|&\leq \|D_x^{\alpha}\overline{u}\|_{L^2(\mathbb{R}^n)}\|D^{\beta}_x (u^2) D^{\alpha-\beta}_x \overline{u}\|_{L^2(\mathbb{R}^n)}\notag\\
&\leq C\|D_x^{\alpha}\overline{u}\|_{L^2(\mathbb{R}^n)}\|D^{\beta}_x (u^2)\|_{L^2(\mathbb{R}^n)}\|D^{\alpha-\beta}_x u\|_{H^{\frac{1}{2}+}}\notag\\
&\leq C\|D_x^{\alpha}\overline{u}\|_{L^2(\mathbb{R}^n)}\|u\|^2_{H^{k-1}}\|D^{\alpha-\beta}_x u\|_{H^{\frac{1}{2}+}}\leq C\|u\|^2_{H^{k-1}}\|D^{\alpha}u\|^2_{L^2(\mathbb{R}^n)}.
\end{align}
Finally,
\begin{align}\label{IntIne3}
\left|\int_{\mathbb{R}^n}u^2D^{\alpha}_x \overline{u} D^{\alpha}_x \overline{u}dx\right|&\leq \|u^2\|_{L^{\infty}(\mathbb{R}^n)}\|D^{\alpha}\overline{u}\|^2_{L^2(\mathbb{R}^n)}\leq C\|u\|^2_{H^{1}}\|D^{\alpha}\overline{u}\|^2_{L^2(\mathbb{R}^n)}\leq C\|u\|^2_{H^{k-1}}\|D^{\alpha}u\|^2_{L^2(\mathbb{R}^n)}.
\end{align}
Using (\ref{PartImag1}), (\ref{IntIne1}), (\ref{IntIne2}) and (\ref{IntIne3}), we obtain
\[
\partial_t\Vert D_x^{\alpha} u(t)\Vert_{L^2(\mathbb{R}^n)}^2\leq C\|u(t)\|^2_{H^{k-1}}\|D^{\alpha}u(t)\|^2_{L^2(\mathbb{R}^n)},                      \ \ \text{for}\ \ k\geq 2.
\]
By Gromwall's inequality we get
\[
\Vert D_x^{\alpha} u(t)\Vert_{L^2(\mathbb{R}^n)}^2\leq C \| D_x^{\alpha}u(0)\|^2_{L^2(\mathbb{R}^n)}\exp\left(C\int_0^t\|u(\tau)\|^2_{H^{k-1}}d\tau\right),\ \ \text{for}\ \ k\geq 2.
\]
In order to obtain global well-posedness in the fractional Sobolev space $H^{s}(\mathbb{R}^n),$ with $s>0$ not an integer, a straightforward argument of nonlinear interpolation theory can be used, which finishes the proof of the theorem.
\endproof
\begin{remark}
Combining the arguments in the proof of Theorem \ref{teo1c} with those in the proof of Theorem \ref{Global1Hs} we can prove that the local solution to the initial value problem (\ref{FoSch})
obtained in the case $\alpha,\beta$ piecewise constant (Theorem \ref{teo1c}), can be extended to $\mathbb{R}.$ 
\end{remark}

\section{Averaging for fast dispersion with nonlinearity $\vert u\vert^2 u$}
In this section we consider the $\epsilon$-scaled equation (\ref{FoSchb}) and analyze the limit as $\epsilon\rightarrow 0^+$, which is also known as the regime of rapidly varying dispersion. Let $\epsilon>0,$ $\beta_\epsilon(t)=\beta(\frac{t}{\epsilon}),$ $\alpha_\epsilon(t)=\alpha(\frac{t}{\epsilon}).$
For $0<\epsilon$, we consider the rescaled problem
\begin{equation}\label{FoSchb_ep}
\left\{
\begin{array}{lc}
i\partial _{t}u^\epsilon+\alpha_\epsilon(t)\Delta u^\epsilon+\beta_\epsilon(t) \Delta^2u^\epsilon+\theta|u^{\epsilon}|^{2}u^{\epsilon}=0, & x\in \mathbb{R}^{n},\ \ t\in \mathbb{R}, \\
u^\epsilon(x,t_0)=\varphi(x), &  x\in \mathbb{R}^{n}, \ \ t_0\in\mathbb{R}.
\end{array}
\right.
\end{equation}
We want to analyze the behavior of the global solution $u^\epsilon$ of (\ref{FoSchb_ep}) as $\epsilon\rightarrow 0^+$ to
the solution $u^0$ of the averaged problem
\begin{equation}\label{FoSchb_aver}
\left\{
\begin{array}{lc}
i\partial _{t}u^0+m(\alpha)\Delta u^0+m(\beta) \Delta^2u^0+\theta|u^0|^{2}u^0=0, & x\in \mathbb{R}^{n},\ \ t\in \mathbb{R}, \\
u^0(x,t_0)=\varphi(x), &  x\in \mathbb{R}^{n}, \ \ t_0\in\mathbb{R},
\end{array}
\right.
\end{equation}
where $m(\alpha)$ and $m(\beta)$ are the averages given by $m(\alpha)=\frac{1}{T_1}\int_0^{T_1}\alpha(r)dr$ and $m(\beta)=\int_0^{1}\beta(r)dr.$  We have the following result.
\begin{theo}\label{lim1}
Let $\varphi\in H^s(\mathbb{R}^n),$ $s>\frac{n}{2},$ and $u^\epsilon, u^0\in C(\mathbb{R},H^s(\mathbb{R}^n))$ be the global mild solutions (\ref{FoSchb_ep}) and  (\ref{FoSchb_aver}), respectively. Then, for all $T>0,$ we have
\[
\lim_{\epsilon\rightarrow 0^+} \Vert u^\epsilon-u^0\Vert_{L^\infty_TH_x^s}= 0.
\]
\end{theo}
\begin{proof}
We define the propagator $U_\epsilon$ associated to the fast dispersion functions $\beta_\epsilon$ and $\alpha_\epsilon$ as
\begin{eqnarray*}
U_\epsilon(t,s)f(x)=
(e^{-i\xi^2A_\epsilon(s,t)+i\xi^4B_\epsilon(s,t)}\widehat{f}(\xi))^{\vee}(x),
\end{eqnarray*}
where $A_\epsilon(s,t)=\int_s^t\alpha_\epsilon(r)dr$ and $B_\epsilon(s,t)=\int_s^t\beta_\epsilon(r)dr.$ In addition, the propagator associated to the averaged dispersion $m(\alpha)$ and $m(\beta)$ is given by
\begin{eqnarray*}
U_0(t,s)f(x)=
(e^{-i\xi^2m(\alpha)(t-s)+i\xi^4m(\beta)(t-s)}\widehat{f}(\xi))^{\vee}(x).
\end{eqnarray*}
Considering the integral formulation associated to (\ref{FoSchb_ep}) and  (\ref{FoSchb_aver}) we have
\begin{align}
u^\epsilon(t,x)-u^0(t,x)&=(U_\epsilon(t,t_0)-U_0(t,t_0))\varphi+i\theta\int_{t_0}^tU_\epsilon(t,\tau)(|u^\epsilon(\tau)|^2u^\epsilon(\tau)-|u^0(\tau)|^2u^0(\tau))d\tau\nonumber\\
&-i\int_{t_0}^t(U_\epsilon(t,\tau)-U_0(t,\tau))(|u^0(\tau)|^2)u^0(\tau))d\tau\nonumber\\
&=(U_\epsilon(t,t_0)-U_0(t,t_0))\varphi+I_1^\epsilon(t,x)+I_2^\epsilon(t,x).\label{aver1}
\end{align}
We will analyze the $H^s$-norm of right hand side of (\ref{aver1}). First of all, notice that $\alpha(t)=m(\alpha)+\alpha_0(t),$ and $\beta(t)=m(\beta)+\beta_0(t),$ where $\alpha_0$ and $\beta_0$ have period $T_1$ and $1$ respectively, and zero mean. Therefore, we get
\begin{eqnarray*}
A_\epsilon(s,t)=\int_s^t\alpha_\epsilon(r)dr=m(\alpha)(t-s)+\epsilon A_\epsilon^\alpha(s,t),\ \ B_\epsilon(s,t)=\int_s^t\beta_\epsilon(r)dr=m(\beta)(t-s)+\epsilon B_\epsilon^\beta(s,t),
\end{eqnarray*}
where $A_\epsilon^\alpha(s,t)=\int_{s/\epsilon}^{t/\epsilon}\alpha_0(\tau)d\tau$ and $B_\epsilon^\beta(s,t)=\int_{s/\epsilon}^{t/\epsilon}\beta_0(\tau)d\tau.$ Therefore,
\begin{eqnarray*}
\Vert U_\epsilon(t,t_0)\varphi-U_0(t,t_0)\varphi\Vert_{H^s}=\Vert \langle\xi\rangle^s e^{\{-i\xi^2m(\alpha)(t-s)+i\xi^4m(\beta)(t-s)\}}\times e^{\{-i\epsilon A_\epsilon^\alpha(s,t)+i\epsilon B_\epsilon^\beta(s,t)-1\}}\widehat{\varphi}\Vert_{L^2(\mathbb{R}^n)}.
\end{eqnarray*}
Since $ A_\epsilon^\alpha(s,t),B_\epsilon^\beta(s,t)\in L^\infty,$ then $\lim\limits_{\epsilon\rightarrow 0}\epsilon A_\epsilon^\alpha(s,t)
=0$ and $\lim\limits_{\epsilon\rightarrow 0}\epsilon B_\epsilon^\beta(s,t)
=0.$ Consequently,
\begin{equation}\label{aver2}
\lim_{\epsilon\rightarrow 0}\sup_{t\in \mathbb{R}}\Vert U_\epsilon(t,t_0)\varphi-U_0(t,t_0)\varphi\Vert_{H^s}=0.
 \end{equation}
Now, we bound the terms $I^\epsilon_1,I^\epsilon_2.$ Notice that for $s> \frac{n}{2}$ it holds
\begin{eqnarray*}
\Vert |u^0|^2u^0\Vert_{H^s}\leq C\sup_{t\in [T-t_0,T+t_0]}\Vert u^0\Vert^3_{H^s}<\infty.
\end{eqnarray*}
Therefore, by working as in the proof of (\ref{aver2}) we get
\begin{equation}
\lim_{\epsilon\rightarrow 0}\sup_{t\in [T-t_0,T+t_0]}\Vert I^\epsilon_2(t,x)\Vert_{H^s}=0.
\end{equation}
From (\ref{grup2a}) we have
\begin{align}
\Vert I^\epsilon_1(t,x)\Vert_{H^s} &\leq  \Vert |u^\epsilon|^2u^\epsilon-|u^0|^2u^0\Vert_{L^1_{T}L^2_x}+\Vert D^s[|u^\epsilon|^2u^\epsilon-|u^0|^2u^0]\Vert_{L^1_{T}L^2_x}\nonumber\\
 &\leq  \Vert |u^\epsilon|^2u^\epsilon-|u^0|^2u^0\Vert_{L^1_{T}L^2_x}+\Vert D^s[(|u^\epsilon|^2-|u^0|^2)u^\epsilon]\Vert_{L^1_{T}L^2_x}\nonumber\\
 &+\Vert D^s[|u^0|^2(u^\epsilon-u^0)]\Vert_{L^1_{T}L^2_x}\nonumber\\
 &:= A_1+A_2+A_3.
 \end{align}
 If $s>\frac{n}{2},$ we get
 \begin{align}
A_1&\leq C\sup_{t\in [T-t_0,T+t_0]}(\Vert u^\epsilon\Vert^2_{L^\infty_x}+\Vert u^0\Vert^2_{L^\infty_x})  \Vert u^\epsilon-u^0\Vert_{L^1_{T}L^2_x}\nonumber\\
&\leq C\sup_{t\in [T-t_0,T+t_0]}(\Vert u^\epsilon\Vert^2_{H^s}+\Vert u^0\Vert^2_{H^s})  \Vert u^\epsilon-u^0\Vert_{L^1_{T}L^2_x}.
\end{align}
By considering $\frac{1}{y_1}+\frac{1}{y_2}=1$ and $s>\frac{n}{2},$ we obtain
\begin{align}
A_2&\leq C\Vert D^s_x(\vert u^{\epsilon}(\tau)\vert^2-\vert u^0(\tau)\vert^2)\Vert_{L^{y_1}_TL^{2}_x}\Vert u^{\epsilon}\Vert_{L^{y_2}_TL^{\infty}_x}\nonumber\\
&+C\Vert \vert u^{\epsilon}(\tau)\vert^2-\vert u^0(\tau)\vert^2\Vert_{L^{y_1}_TL^{\infty}_x}\Vert D^s_x u^{\epsilon}\Vert_{L^{y_2}_TL^{2}_x}\nonumber\\
&\leq C\Vert u^{\epsilon} (D^s_x(u^{\epsilon}-u^0))\Vert_{L^{y_1}_TL^{2}_x}\Vert u^{\epsilon}\Vert_{L^{y_2}_TL^{\infty}_x}+C\Vert (u^{\epsilon}-u^0) D^s_xu^0\Vert_{L^{y_1}_TL^{2}_x}\Vert u^{\epsilon}\Vert_{L^{y_2}_TL^{\infty}_x}\nonumber\\
&+C\Vert  u^{\epsilon}-u^0\Vert_{L^{2y_1}_TL^{\infty}_x}(\Vert  u^{\epsilon}\Vert_{L^{2y_1}_TL^{\infty}_x}+\Vert u^0\Vert_{L^{2y_1}_TL^{\infty}_x})\Vert D^s_x u^{\epsilon}\Vert_{L^{y_2}_TL^{2}_x}\nonumber\\
&\leq C\Vert D^s_x(u^{\epsilon}-u^0)\Vert_{L^{2y_1}_TL^{2}_x}\Vert u^{\epsilon} \Vert_{L^{2y_1}_TL^{\infty}_x}\Vert u^{\epsilon}\Vert_{L^{y_2}_TL^{\infty}_x}+C\Vert u^{\epsilon}-u^0\Vert_{L^{y_1}_TL^{\infty}_x}\Vert D^s_xu^0\Vert_{L^{\infty}_TL^{2}_x}\Vert u^{\epsilon}\Vert_{L^{y_2}_TL^{\infty}_x}\nonumber\\
&+C\Vert  u^{\epsilon}-u^0\Vert_{L^{2y_1}_TL^{\infty}_x}(\Vert  u^{\epsilon}\Vert_{L^{2y_1}_TL^{\infty}_x}+\Vert u^0\Vert_{L^{2y_1}_TL^{\infty}_x})\Vert D^s_x u^{\epsilon}\Vert_{L^{y_2}_TL^{2}_x}\nonumber\\
&\leq C(\Vert D^s_x(u^{\epsilon}-u^0)\Vert_{L^{2y_1}_TL^{2}_x}+\Vert  u^{\epsilon}-u^0\Vert_{L^{2y_1}_TL^{\infty}_x})\nonumber\\
&\leq C\Vert  u^{\epsilon}-u^0\Vert_{L^{2y_1}_TH^{s}}.
\end{align}
and
\begin{align}
A_3&\leq C\Vert D^s(\vert u^0\vert^2)\Vert_{L^{y_2}_TL^{2}_x}\Vert  u^\epsilon-u^0\Vert_{L^{y_1}_TL^{\infty}_x}+C\Vert \vert u^0\vert^2\Vert_{L^{y_2}_TL^{\infty}_x}\Vert D^s( u^\epsilon-u^0)\Vert_{L^{y_1}_TL^{2}_x}\nonumber\\
&\leq C(\Vert  u^\epsilon-u^0\Vert_{L^{y_1}_TL^{\infty}_x}+\Vert D^s( u^\epsilon-u^0)\Vert_{L^{y_1}_TL^{2}_x})\nonumber\\
&\leq C\Vert  u^\epsilon-u^0\Vert_{L^{y_1}_TH^{s}}.\label{aver5}
\end{align}
Therefore, from (\ref{aver1})-(\ref{aver5}) we get
\begin{eqnarray}
\Vert  u^\epsilon-u^0\Vert_{L^{\infty}_TH^{s}}\leq C_\epsilon+C_1\Vert  u^\epsilon-u^0\Vert_{L^{2y_1}_TH^{s}}.\label{aver6}
\end{eqnarray}
By using the Lemma A.1 in Cazenave and Scialom \cite{caze_scial} and (\ref{aver6}), we get that there exists a positive constant $K=K(C_1,y_1,T)$ such that
\[
\Vert  u^\epsilon-u^0\Vert_{L^{\infty}_TH^{s}}\leq KC_\epsilon\rightarrow 0,\ \mbox{as}\ \epsilon\rightarrow 0^+,
\]
which complete the proof.
\end{proof}

\textbf{Acknowledgements:} The second author was partially supported by Fondo Nacional de Financiamiento para la Ciencia, la Tecnolog\'{\i}a y la Innovaci\'on Francisco Jos\'e de Caldas, contrato Colciencias FP 44842-157-2016.

\end{document}